\documentclass{birkmult}
 \newtheorem{thm}{Theorem}[section]
 
 \newtheorem{lem}[thm]{Lemma}
 
 \theoremstyle{defn}
 \newtheorem{defn}[thm]{Definition}
 \theoremstyle{remark}
 
 \newtheorem*{ex}{Example}
 \numberwithin{equation}{section}

\begin{document}
%-------------------------------------------------------------------------
% editorial commands: to be inserted by the editorial office
%
%\firstpage{1}
%\volume{228}
%\DOI{003-0001}
%\seriesextraline{This is the Concrete Title of this Book\br H.E. R and S.T.C. W, Eds.}
%\firstpage{1}
\title[$4n + 2$ Simple permutations, Pasting and Reversing]{Simple permutations with order $4n + 2$ by means of Pasting and Reversing}
%----------Author 1
\author[P. Acosta-Hum\'anez]{Primitivo B. Acosta-Hum\'anez}

\address{
Universidad del Atl\'antico \& INTELECTUAL.CO\\
Barranquilla\\
Colombia}

\email{primi@intelectual.co}

\thanks{The first author is partially supported by the MICIIN/FEDER grant number MTM2012-31714 Spanish Government, also by ECOS Nord France-Colombia No C12M01
and Colciencias.}
%----------Author 2
\author[O. Mart\'{i}nez-C.]{\'Oscar E. Mart\'{i}nez-Castiblanco}
\address{Universidad Sergio Arboleda\\
Bogot\'a D.C\\
Colombia}
\email{oscare.martinez@usa.edu.co}


\date{May 10, 2015}
\dedicatory{Dedicated to Jes\'us Hernando P\'erez (Pelusa), teacher and friend.}
%%% ----------------------------------------------------------------------

\begin{abstract}
The problem of genealogy of permutations has been solved partially by Stefan (odd order) and Acosta-Hum\'anez \& Bernhardt (power of two). It is well known that Sharkovskii's theorem shows the relationship between the cardinal of the set of periodic points of a continuous map, but simple permutations will show the behaviour of those periodic points. Recently Abdulla et al studied the structure of minimal  $4n+2$-orbits of the continuous endomorphisms on the real line. This paper studies some combinatorial dynamics structures of permutations of mixed order $4n+2$, describing its genealogy, using Pasting and Reversing.
\end{abstract}

%%% ----------------------------------------------------------------------
\maketitle
%%% ----------------------------------------------------------------------
%\tableofcontents
\section*{Introduction}
According to the Encyclop\ae dia Brittanica, genealogy is ``the study of family origins and history".  In this paper, which is an slightly improvement of \cite{acma,acma1},  we translate this idea to the context of combinatorial dynamics. In particular, we study how the concept of genealogy can be used as a way to understand forcing relations between periodic points. Furthermore, we translate the problems related to periodic points of functions to the framework of group theory too. That is, we study algebraically the dynamics through $n$-cycles and permutations, using Pasting and Reversing of permutations and $n$-cycles, according to \cite{Genealogia,acma,acma1}.

Combinatorial Dynamics as a field appears in 1964 with the paper ``Co-Existence of Cycles of a Continuous Mapping of a Line onto Itself"  written by Oleksandr Mikolaiovich Sharkovskii (see \cite{Sharkovskii,Sharkovskii2}). From this point and on, the study of algebraic and topological relations of continuous functions in $\mathbb{R}$ becomes important. In this context, permutations can be used to show minimal orbits, where a periodic orbit is called minimal if it is minimal period of the map in Sharkovski's ordering.

The structure of minimal odd orbits was characterized by Stefan in 1977, see \cite{Stefan}. 
To characterize the structure of general minimal orbits the notions of simple and strongly simple orbit
was introduced and studied by Louis Block, see \cite{Block1,Block2,Block4}. In this definition Block emphasizes in the three different kinds of orbits related to the three different ``tails" in Sharkovskii's order: the left tail which corresponds to odd order simple permutations (also known as Stefan orbits), the right tail which corresponds to power of two order simple permutations (due to Block, see \cite{Block2}) and the middle tail which corresponds to mixed order simple permutations (see also \cite{Minimal,Coppel,Ho}). Moreover, Alseda et al., see \cite{Minimal}, and Block \& Coppel, see \cite{Block3} independently proved that minimal orbits are strongly simple. 

Chris Bernhardt suggested the study of predecessors and successors of a simple permutation in his paper ``Simple permutations with order a power of two" (see \cite{Bernhardt}) in which he describes a procedure to find the predecessor and the successor of any permutation of order a power of two by using transpositions, rising up a tree with this partial ordering. In the case of Stefan orbits, there exists two simple permutations per order and its forcing (genealogy) can be described as two separated lines according to the first permutation in an explicit way, while for other kind of simple permutation is more complicated the obtaining of an explicit genealogy.

Inspired and motivated by P. Mumbr\'u, the first author gave a new perspective to Bernhardt's results, by including the operations Pasting and Reversing in order to obtain a recursive algorithm that produces the genealogy lines in order a power of two (see \cite{Genealogia}). Thus, Pasting and Reversing becomes an important way to study the genealogy problem in the remaining order: Mixed Order. In this way, in 2010, the authors presented a first draft of an algebraic and combinatorial dynamics study of simple permutations with order $4n+2$, see \cite{acma}, followed by the preprint \cite{acma1} in 2011, in where the dynamic is not considered. Curiously, Alsed\'a et. al. \cite{Minimal} suggest a sort of \emph{reversing} in the study of some periodic orbits, which it was not considered here.

Recently Abdulla et. al., see \cite{abdulla}, presented a simple and constructive proof of the results obtained by Alsed\'a et. al. and Block \& Coppel, see \cite{Minimal, Coppel}, on the structure of minimal  $4n+2$-orbits of the continuous endomorphisms on the real line. They proved that there are four types of Markov graphs up to inverse graphs. Some of these results coincide with some results obtained previously in \cite{acma,acma1} through pasting and reversing techniques, obtained in an independent way.

In this paper, inspired by  \cite{Genealogia, Bernhardt}, we improve the previous results presented in \cite{acma,acma1}. That is, following the same spirit of \cite{Genealogia, Bernhardt}, we will show a procedure to construct simple permutations by using Pasting and Reversing, Markov graphs and primitive functions. We follow \cite{Bernhardt, Block3, Ho} for simple permutations, for instance strong simple periodic orbits will not be considered here.

\section{Combinatorial Dynamics}

The theoretical background of this section is based on references \cite{Combinatorial,Minimal,Bernhardt,Block1,Block4,Ho, Misiurewicz}. The main topics that we need correspond to Sharkovskii's theorem,  primitive functions and Markov graphs.  From a basic course of algebra  it is well known that
$(S_{n},\circ)$ denotes the group of permutations of order ${n}$, which corresponds to the bijective functions over finite subsets of natural numbers. We denote by $C_n$ the set of $n-$cycles of $S_n$, that is, $\sigma\in C_n$ means that $\sigma$ has length $n$ and for instance $\sigma^n=e$ and $\sigma^{n-k}\neq e$, being $0<k<n$. A partition of the interval $J=[x_1,x_n]$  in ${n-1}$ closed subintervals will be defined as
\begin{equation*}
P_{n}=\{x_{i},x_{i+1} \in J : x_{i} < x_{i+1}, \forall i = 1,\dots,n-1 \}
\end{equation*}
and the closed subintervals associated to this partition are denoted by $$J_k:=[x_k,x_{k+1}],\,\, 1\leq k\leq n-1.$$ In particular, if $\forall k\in\mathbb{Z}^+, x_k=k$, then we say that $P_n$ is the \emph{natural partition} of $J$. Thus, we arrive to the following definition.
  \begin{defn}\label{def1}
   Consider $m,n,i\in\mathbb{Z}^+$ , $J=\{x\in \mathbb{Z}:1\leq x\leq mn\}$. The $i-$th natural partition of $J$ of size $n$ is given by $$P(nm,m,i):=\left\{x\in\mathbb{Z}: (i-1)n\leq x\leq in,\, i\leq m\right\}.$$
    \end{defn}

		\begin{thm}[Sharkovskii's Theorem, \cite{Sharkovskii}]\label{th1}
 Let $f$ be a continuous map from the real line to itself and let $\lhd$ be the ordering of the positive integers
\begin{equation}\label{shark}
3 \lhd 5 \lhd 7  \lhd \dots \lhd 3\cdot2 \lhd 5\cdot2 \lhd \dots \lhd 3\cdot2^{2} \lhd 5\cdot2^{2} \lhd \dots \lhd 2^{3} \lhd 2^{2} \lhd 2 \lhd 1.\end{equation} If $f$ has a periodic point of period $n$ and $n$ satisfies $n\lhd m$, then $f$ has a periodic point of period $m$.
		\end{thm}
The main result of Sharkovskii relates a new order of the natural numbers and the existence of cycles (periodic points). He defined a new order for the set of natural numbers as follows: $n_{1}$ precedes $n_{2}$ ($n_{1} \lhd n_{2}$) if for all continuous map of the line into itself the existence of a cycle of order $n_{1}$ implies the existence of a cycle of order $n_{2}$. We recall that by minimal orbit we mean the periodic orbit which is minimal period of the map in Sharkovski's ordering \eqref{shark}. The following definitions are adapted from \cite{Bernhardt}.
 
 \begin{defn}\label{def2}
A permutation $\theta$ belongs to the set of permutations of $f$ (named $Perm(f)$) if and only if there exists a partition $P_{k}$ such that $f(x_{i})=x_{\theta_k(i)}$, for all $x_i\in P_k$. It means,
\begin{equation*}
Perm(f)=\{\theta_k\in S_k:\exists P_k \wedge f(x_{i})=x_{\theta_k(i)}, \forall x_i\in P_k \}
\end{equation*}
    \end{defn}

In order to set a relation between permutations and Sharkovskii's Theorem, it is necessary to define a relation between permutations of different orders, which is known as \emph{forcing}, see \cite{Combinatorial}.

    \begin{defn}\label{def3}
Let $\theta$ and $\eta$ be permutations and consider the sets $$F_{\theta}:=\{f:\theta\in Perm(f)\},\quad F_{\eta}:=\{f:\eta\in Perm(f)\}.$$ We say that $\theta$ dominates to $\eta$, denoted by $\theta \lhd \eta$, whether $F_{\theta}\subset F_{\eta}$.
    \end{defn}

Without loss of generality, we can consider linear piece-wise differentiable functions defined on $\mathbb{R}$ ($\overline{f}$ is continuous on $\mathbb{R}$), due to the dynamics only requires the study of periodic points. 

    \begin{defn}\label{def4} Consider a partition $P_n$ and a permutation $\theta\in S_{n}$. A function associated to  $\theta$ is a linear piece-wise differentiable function $\overline{f}$ satisfying 
$$\overline{f}(x)= \left \{ \begin{array}{ccc}
             x_{\theta_1} & if & x < x_1 \\
             \\ m_1x+b_1 &  if & x_1 \leq x <x_2 \\
             \vdots&&
             \\ m_{n-2}x+b_{n-2} &  if & x_{n-2} \leq x < x_{n-1} \\
              \\ m_{n-1}x+b_{n-1} &  if & x_{n-1} \leq x < x_{n} \\
             \\ x_{\theta_n} &  if  & x \geq x_n\end{array}\right.$$
where $$m_k={x_{\theta(k+1)}-x_{\theta(k)}\over x_{k+1}-x_{k}},\quad b_k=x_{\theta(k)}-m_kx_k,\quad k=1,\ldots,n-1.$$ We say that $\overline{f}$ is the primitive function of $\theta$ whether $P_n$ is the natural partition $P_n=\{1,\ldots,n\}$.
    \end{defn}

\noindent From now on, we work with primitive functions of permutations belonging to $C_n$.

\begin{defn}\label{def5}
Consider a partition $P_n$ and $\theta \in Perm(f)$. The Markov graph associated to $f$ and $\theta$ is a directed graph with $n-1$ vertices $J_{1},J_{2},\ldots,J_{n-1}$ such that an arrow is drawn from $J_{k}$ to $J_{l}$ if and only if $f(J_k)\supseteq J_l$. The Markov graph associated to $\theta$ and its primitive function will be called the Markov graph of $\theta$.
\end{defn}

Markov graphs of $\theta$ are also known as A-graphs, see \cite{Combinatorial, Bernhardt,Ho} among others. Recently Abdulla et. al., see \cite{abdulla}, introduced Markov graphs with red edges that they called digraphs, which will not be considered in this work, digraphs have the strong condition that an arrow is traced from $J_l$ to $J_k$ whether $J_k$ is the only interval contained in $J_l$. Markov graphs are useful to study the dynamics of a function due to we can see periodic points whether $f^n(J_k)\supset J_k$, i.e.,  departing from $J_k$ there are $n$ arrows to come back to $J_k$. Sharkovskii's Theorem gives information about the existence of periodic points. If we know the order of the period, then Markov graph of $\theta$ allow us to find information about the structure of those periods.

The following examples will illustrate the previous definitions.

    \begin{ex}
Consider the permutation $$\theta=\left(\begin{array}{cccc}
1 & 2 & 3 & 4\\
2 & 3 & 4 & 1\end{array}\right).$$ The primitive function $\overline{f}$ associated to $\theta$ is:

$$f(x)= \left \{ \begin{array}{ccc}
             2 & if & x < 1 \\
             \\ x+1 &  if & 1 \leq x < 3 \\
             \\ 13-3x &  if & 3 \leq x < 4 \\
             \\ 1 &  if  & x \geq 4\end{array}\right.$$

\begin{figure}[h]
\noindent \begin{centering}
\includegraphics[scale=0.5]{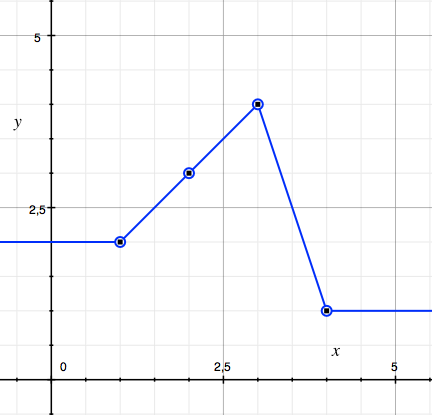}
\par\end{centering}

\caption{Primitive Function, Example 1}

\end{figure}

The intervals used to construct the Markov graph related to $\overline{f}$ and $\theta$ in Example 1 are $J_{1}=[1,2]$, $J_{2}=[2,3]$, $J_{3}=[3,4]$. We can notice that $f(J_{1}) \supseteq J_{2}$, $f(J_{2}) = J_{3}$ and $f(J_{3})=J_{1} \cup J_{2} \cup J_{3}$.

\begin{figure}[h]
\noindent \begin{centering}
\includegraphics[scale=0.5]{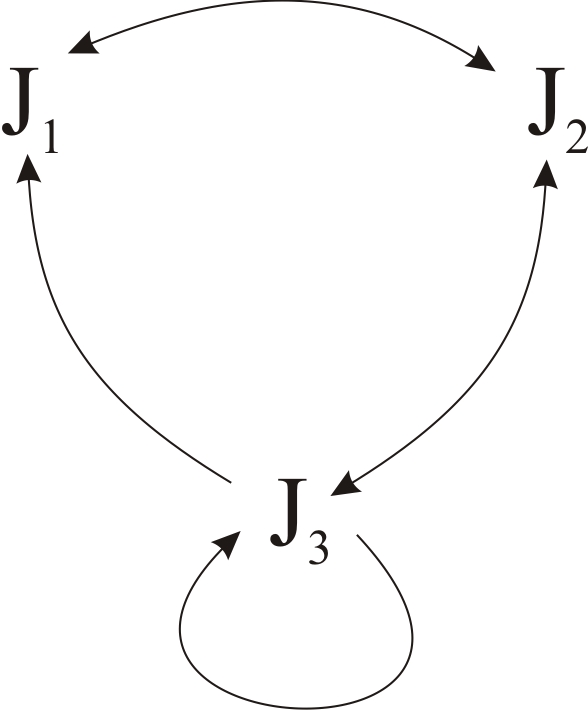}
\par\end{centering}

\caption{Markov Graph, Example 1}

\end{figure}

\noindent According to Sharkovskii's Theorem the existence of a periodic point of order 4 implies the existence of periodic points of order 2 and 1. However, Markov graphs allow us to find a periodic points of order 3.
    \end{ex}

  \begin{ex}
Consider the permutation $$\theta=\left(\begin{array}{cccccc}
1 & 2 & 3 & 4 & 5 & 6\\
6 & 4 & 5 & 1 & 2 & 3\end{array}\right).$$ The primitive function $\overline{f}$ associated to $\theta$ is:

$$f(x)= \left \{ \begin{array}{ccc}
             6 & if & x < 1 \\
	\\ -2x+8 &  if & 1 \leq x < 2 \\
	\\ x+2 &  if & 2 \leq x < 3 \\
	\\ -4x+17 &  if & 3 \leq x < 4 \\
	\\ x-3 &  if & 4 \leq x < 6 \\
	\\ 3 &  if  & x \geq 6\end{array}\right.$$

\begin{figure}[h]
\noindent \begin{centering}
\includegraphics[scale=0.5]{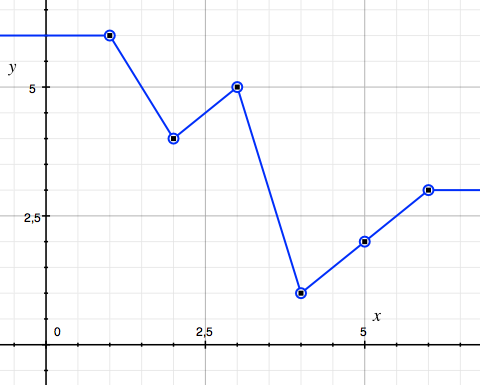}
\par\end{centering}

\caption{Primitive Function, Example 2}

\end{figure}

The intervals used to construct the Markov graph related to $\overline{f}$ and $\theta$ in Example 2 are $J_{1}=[1,2]$, $J_{2}=[2,3]$, $J_{3}=[3,4]$, $J_{4}=[4,5]$, $J_{5}=[5,6]$. We can notice that
$f(J_{2}) = J_{4}$, $f(J_{4}) = J_{1}$, $f(J_{5}) = J_{2}$, $f(J_{1}) \supseteq J_{4}$, $f(J_{1}) \supseteq J_{5}$, $f(J_{3}) \supseteq J_{1}$, $f(J_{3}) \supseteq J_{2}$, $f(J_{3}) \supseteq J_{3}$ and $f(J_{3}) \supseteq J_{4}$.

\begin{figure}[h]
\noindent \begin{centering}
\includegraphics[scale=0.4]{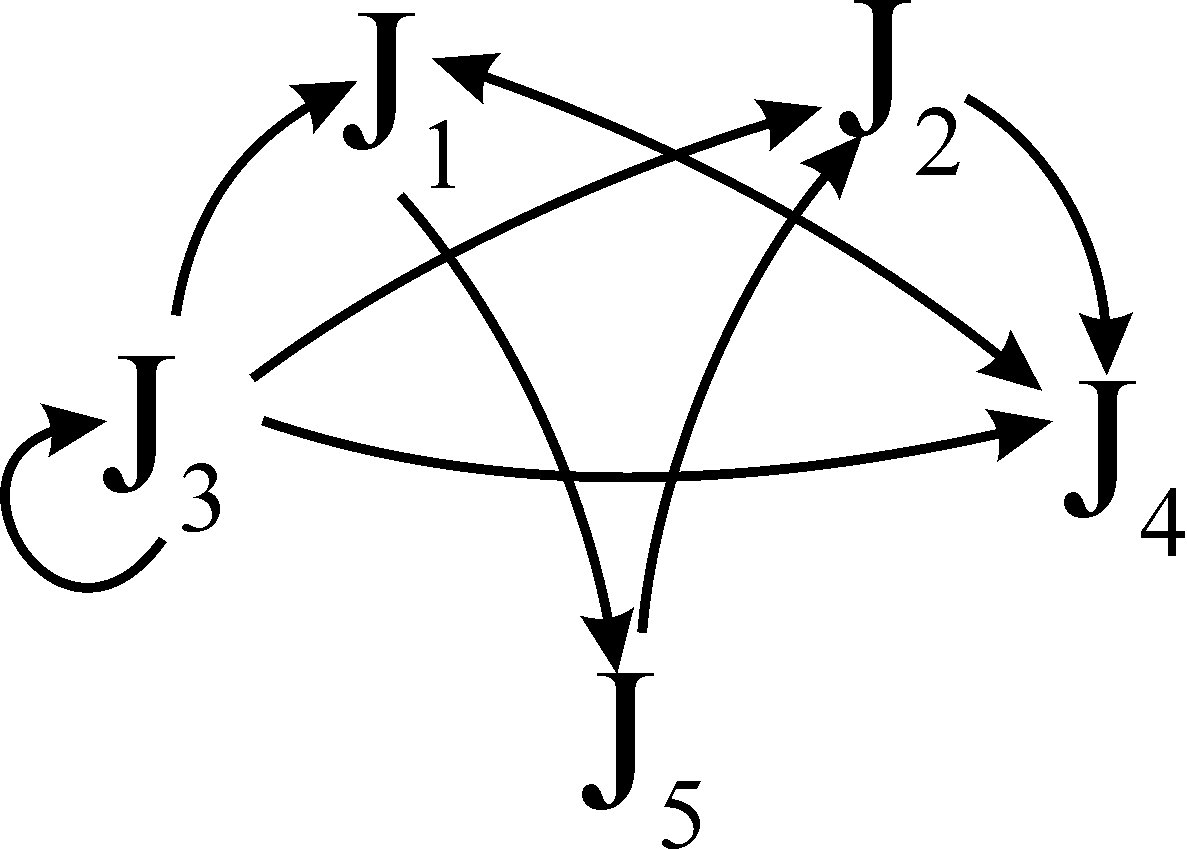}
\par\end{centering}

\caption{Markov Graph, Example 2}
\end{figure}
    \end{ex}
The following results due to Bernhardt, \cite{Bernhardt}, are derived from Definition \ref{def4}.
\begin{lem}\label{lem1}
Consider $\theta\in C_n$, $\eta\in C_m$, $\theta\neq \eta$. Let $\overline{f}$ be the primitive function of $\theta$. If $\eta\in Perm(\overline{f})$, then  the Markov graph of $\theta$ has a non-repetitive loop of length $m$ corresponding to $\eta$.
\end{lem}

\begin{lem}\label{lem2}
Consider $\theta\in C_n$, $\eta\in C_m$, $\theta\neq \eta$. The Markov graph of $\theta$ has a non-repetitive loop of length $m$ corresponding to $\eta$ if and only if $\theta\lhd \eta$.
\end{lem}

Lemmas \ref{lem1} and \ref{lem2} lead to the following result, see \cite{Bernhardt}.

\begin{thm}[Sharkovskii's Extended Theorem, \cite{Bernhardt}]\label{th2}
If $\theta\in C_n$, then for any integer $m$ satisfying $n\lhd m$ there exists $\eta\in C_m$ such that $\theta\lhd \eta$.
\end{thm}

Following \cite{Bernhardt,Block1,Block2,Block4} we give the following definition.
   	\begin{defn}\label{def6} A permutation is considered simple whether it satisfies one of the following conditions according to its order\medskip

	\begin{enumerate}

	\item Odd order. Consider $n\in \mathbb{Z}^+$ The permutation $\theta\in C_{2n+1}$ is simple whether $\theta\in\{\alpha_{2n+1},\beta_{2n+1}\}$, where
\begin{eqnarray*}
\alpha_{2n+1} & = & (1,2n+1,n+1,n,n+2,n-1,n+3,\dots,2,2n)\\
 & = & \left(\begin{array}{ccccccccc}
1 & 2 & \dots & n & n+1 & n+2 & \dots & 2n & 2n+1\\
2n+1 & 2n & \dots & n+2 & n & n -1& \dots & 1 & n+1\end{array}\right)\
\end{eqnarray*}
and
\begin{eqnarray*}
\beta_{2n+1} & = & (1,n+1,n+2,n,n+3,n-1,\dots,2,2n+1)\\
 & = & \left(\begin{array}{ccccccccc}
1 & 2 & \dots & n & n+1 & n+2 & \dots & 2n & 2n+1\\
n+1 & 2n+1 & \dots & n+3 & n+2 & n & \dots & 2 & 1\end{array}\right)\
\end{eqnarray*}\medskip

	\item
	\emph{(Order a power of two)}
A permutation $\theta\in C_{2^n}$ is simple whether it satisfies the following two conditions
\begin{enumerate}
\item
$\theta^{2^{j}}[P(2^{n},2^{j},i)]=P(2^{n},2^{j},i)$
\item
$\theta^{2^{j}}[P(2^{n},2^{k+1},j)]\cap P(2^{n},2^{k+1},j)=\emptyset$
\end{enumerate}\medskip

	\item\medskip

	\emph{(Mixed order $(2k+1)2^{s}$)}
A permutation $\theta\in C_{(2k+1)2^{s}}$, where $k>1$ and $s>1$, is simple whether it satisfies the following two conditions:

\begin{enumerate}
\item $\theta\left[P\left((2k+1)2^{s},2^{s},j\right)\right]=P\left((2k+1)2^{s},2^{s},\sigma\left(j\right)\right)$, where
$\sigma\in C_{2^{s}}$ is simple;\medskip

\item $\theta^{2^{s}}$ restricted to each $P\left((2k+1)2^{s},2^{s},j\right)$  is simple for all $j$, that is, the set $\theta^{2^{s}}(P\left((2k+1)2^{s},2^{s},j\right))=P\left((2k+1)2^{s},2^{s},j\right)$ induces a simple permutation for all $j$.
\end{enumerate}

	\end{enumerate} The set of simple permutations belonging to $C_k$ is denoted by $Sim(k)$.
	\end{defn}

The following examples show some permutations corresponding to the previous definition.

\begin{ex} The permutations
$$\alpha_{5}=\left(\begin{array}{ccccc}
1 & 2 & 3 & 4 & 5\\
5 & 4 & 2 & 1 & 3\end{array}\right),\quad
\beta_{5}=\left(\begin{array}{ccccc}
1 & 2 & 3 & 4 & 5\\
3 & 5 & 4 & 2 & 1\end{array}\right)$$ are the simple permutation of order $5$ (Stefan orbits of order $5$), that is,  $Sim(5)=\{\alpha_5,\beta_5\}$.
\end{ex}

\begin{ex}The permutation
$$\theta=\left(\begin{array}{cccc}
1 & 2 & 3 & 4\\
3 & 4 & 2 & 1 \end{array}\right)$$ is a simple permutation of order $4$ because $\theta\{1,2\}=\{3,4\}$, $\theta\{3,4\}=\{1,2\}$, $\theta\{1,2,3,4\}=\{1,2,3,4\}$, $\theta^{2}\{1,2\}=\{1,2\}$, $\theta^{2}\{3,4\}=\{3,4\}$, $\theta^{2}\{1\}=\{2\}$,  $\theta^{2}\{2\}=\{1\}$,  $\theta^{2}\{3\}=\{4\}$,  $\theta^{2}\{4\}=\{3\}.$ That is, $\theta\in Sim(4)$.
\end{ex}

\begin{ex}
Consider the permutation $$\theta=\left(\begin{array}{cccccc}
1 & 2 & 3 & 4 & 5 & 6\\
6 & 5 & 4 & 1 & 3 & 2\end{array}\right).$$

Taking $$\alpha_3=\left(\begin{array}{ccc}
1 & 2 & 3\\
3 & 1 & 2\end{array}\right),\quad \beta_3=\left(\begin{array}{ccc}
1 & 2 & 3\\
2 & 3 & 1\end{array}\right), \quad \sigma=\left(\begin{array}{cc}
1 & 2\\
2 & 1\end{array}\right),$$ we see that
 \begin{eqnarray*}
\theta\left[P\left(6,2,1\right)\right] & = & \theta\left[\left\{ 1,2,3\right\} \right]=\left\{ 6,5,4\right\} =\left\{ 4,5,6\right\} =P\left(6,2,\sigma\left(1\right)\right)=P\left(6,2,2\right),\end{eqnarray*}
\begin{eqnarray*}
\theta\left[P\left(6,2,2\right)\right] & = & \theta\left[\left\{ 4,5,6\right\} \right]=\left\{ 1,3,2\right\} =\left\{ 1,2,3\right\} =P\left(6,2,\sigma\left(2\right)\right)=P\left(6,2,1\right).\end{eqnarray*} Due to 
$$\theta^{2}=\left(\begin{array}{cccccc}
1 & 2 & 3 & 4 & 5 & 6\\
2 & 3 & 1 & 6 & 4 & 5\end{array}\right),$$
we have that  \begin{eqnarray*}
\theta^2\left[P\left(6,2,1\right)\right] & = & \theta^2\left[\left\{ 1,2,3\right\} \right]=\left\{ 2,3,1\right\} =\left\{ 1,2,3\right\} =P\left(6,2,1\right),\end{eqnarray*}
\begin{eqnarray*}
\theta^2\left[P\left(6,2,2\right)\right] & = & \theta^2\left[\left\{ 4,5,6\right\} \right]=\left\{ 6,4,5\right\} =\left\{ 4,5,6\right\} =P\left(6,2,2\right),\end{eqnarray*} which induce the simple permutations $\alpha_3$ and $\beta_3$. That is, $\theta^{2}$ restricted to $P\left(6,2,1\right)$  is $\beta_3$ and $\theta^{2}$ restricted to $P\left(6,2,2\right)$  is $\alpha_3$.
\end{ex}

The concept of $(2k-1)2^s$ simple periodic orbit was refined by  Block \& Coppel in \cite{Block3}, and adapted from \cite{abdulla}, as follows:

\begin{defn}\label{def7} 
A simple $(2k-1)2^s$-orbit of $f$ is said to be strongly simple whether $f$ maps the midpoint of every block of $2k-1$ consecutive points into another such midpoint, with one exception. In particular, this implies that $f$ maps every block monotonically onto another block, with one exception.
\end{defn}

The following result was obtained independently by Alsed\'a et. al. and Block \& Coppel, see \cite{Minimal, Block3}.
\begin{thm}[Alsed\'a et. al. - Block \& Coppel, \cite{Minimal, Block3}]\label{th3}
Minimal $(2k-1)2^s$ -orbit of $f$, where $k > 1$
and $s > 1$, is strongly simple unless $k = 2$, in which case it is simple.
\end{thm}

\section{Pasting and Reversing}

Pasting and Reversing operations have been defined and applied on integer numbers (\cite{Pegamiento}), rings over commutative fields (\cite{preprint}), vector spaces and matrices (\cite{Vector,preprint}), cycles and permutations (\cite{Genealogia,acma,acma1}).  A kind of reversing was mentioned by Alsed\'a et. al (\cite{Minimal}) to study minimal orbits. In this paper, Pasting and Reversing operations will be used to describe the genealogy of simple permutations with mixed order $4n+2$  and a procedure to generate them  such as were done in the case of simple permutations with order a power of two, see \cite{Genealogia}.

	\begin{defn}[Pasting of Disjoint Cycles of a Permutation]\label{def8}
Consider $\theta =uv\in S_{n}$, being $u$ and $v$ disjoint cycles such that $$\theta =uv=(i_{1}, \dots, i_{k})(i_{k+1}, \dots, i_{n}),$$ where $i_1=1$ and for  $j>k+1$, it satisfies that $ i_{k+1}<i_j$. The Pasting of $u$ with $v$ is the $n$-cycle defined as follows:
\begin{equation*}
u\diamond v=(i_{1},\dots,i_{k},i_{k+1},\dots,i_{n})
\end{equation*}
\end{defn}

\begin{defn}[Reversing of Cycles]\label{def9}
Consider a $q$-cycle of a permutation $\theta\in S_n$ given by $u=(i_{1},\dots,i_{q})$, where $i_1<i_j$ for all $j>1$ and $q\leq n$. The Reversing of $u$, denoted by $\widetilde{u}$, is 
\begin{equation*}
\widetilde{u}:=(i_1,i_{k},i_{k-1},\dots,i_{2}).
\end{equation*}
	\end{defn}
Recall that we write any $q$-cycle in the form $(i_1,i_2,\ldots,i_q)$ where $i_1<i_j$ for all $j>1$. Furthermore, pasting and reversing of permutations written as cycles will not be considered here. Pasting and Reversing of disjoint cycles of a permutation are considered in Definition \ref{def8} and Definition \ref{def9}, whilst Pasting and Reversing of permutations will be given in Definition \ref{def10} and Definition \ref{def11}. 
	\begin{defn}[Pastings of Permutations]\label{def10}
Consider $\alpha\in S_{m}$, $\beta\in S_{n}$. The left and right pastings of $\alpha$ with $\beta$ are permutations in $S_{n+m}$ defined as follows: \medskip

\begin{itemize}
\item
Left Pasting of $\alpha$ with $\beta$, denoted by $\alpha\mid\diamond\beta$, is
\begin{eqnarray*}
\alpha\mid\diamond\beta & := & \left(\begin{array}{cccccc}
1 & \dots & m & m+1 & \dots & m+n\\
\alpha (1)+n & \dots & \alpha (m)+n & \beta (1) & \dots & \beta (n)\end{array}\right)\
\end{eqnarray*}\medskip

\item
Right Pasting of $\alpha$ with $\beta$, denoted by $\alpha\diamond\mid\beta$, is
\begin{eqnarray*}
\alpha\diamond\mid\beta & =: & \left(\begin{array}{cccccc}
1 & \dots & m & m+1 & \dots & m+n\\
\alpha (1) & \dots & \alpha (m) & \beta (1)+m & \dots & \beta (n)+m\end{array}\right).\
\end{eqnarray*}
\end{itemize}
\end{defn}

	\begin{defn}[Reversing of Permutations]\label{def11}
Consider $\alpha\in S_{m}$. Reversing of $\alpha$, denoted by $\widetilde{\alpha}$,  is the permutation of $S_{m}$ defined as
\begin{eqnarray*}
\widetilde{\alpha} & = & \left(\begin{array}{ccccc}
1 & 2 & \dots & m-1 & m\\
\alpha (m) & \alpha(m)-1 & \dots & \alpha (2) & \alpha (1)\end{array}\right).\
\end{eqnarray*}
	\end{defn}

The following results, which appear without proof in \cite{Genealogia}, states the properties as algebraic structure of Pasting and Reversing for permutations and disjoint cycles of permutations. Here we prove these theorems.

	\begin{thm}[Idempotency Laws of Reversing]\label{th4}
Assume $\alpha\in S_{m}$, $\theta \in S_n$ such that $u$ is a $q$-cycle of $\theta$ given by $u=(i_{1},\dots,i_{q})$, then the following hold
\begin{enumerate}
\item
$\widetilde{\widetilde{\alpha}}=\alpha$\\
\item
$\widetilde{\widetilde{u}}=u$
\end{enumerate}
	\end{thm}

	\begin{proof}
We prove the theorem according to each item.
	\begin{enumerate}
	\item
Consider $\alpha\in S_{m}$
\begin{eqnarray*}
\alpha & =  & \left(\begin{array}{ccccc}
1 & 2 & \dots & m-1 & m\\
\alpha (1) & \alpha (2) & \dots & \alpha (m-1) & \alpha (m)\end{array}\right)\\
\widetilde{\alpha} & = & \left(\begin{array}{ccccc}
1 & 2 & \dots & m-1 & m\\
\alpha (m) & \alpha (m-1) & \dots & \alpha (2) & \alpha (1)\end{array}\right)\\
\widetilde{\widetilde{\alpha}} & =  & \left(\begin{array}{ccccc}
1 & 2 & \dots & m-1 & m\\
\alpha (1) & \alpha (2) & \dots & \alpha (m-1) & \alpha (m)\end{array}\right)\\
\widetilde{\widetilde{\alpha}} & =  & \alpha\
\end{eqnarray*}\medskip

\item
Consider $u=(i_{1},i_2,\dots,i_{q-1},i_{q})$. Therefore
\begin{eqnarray*}
%(u) & = & (i_{1},i_{2},\dots,i_{s_{1}})\\
\widetilde{u} & = & (i_1,i_{q},i_{q-1}\dots,i_{2})\\
\widetilde{\widetilde{u}} & = & (i_{1},i_{2},\dots,i_{q-1},i_q)\\
\widetilde{\widetilde{u}} & = & u\\\end{eqnarray*}
	\end{enumerate}
Thus the theorem is proven.
	\end{proof}

	\begin{thm}[Associative Property of Pasting]\label{th5}
Consider the permutations $\alpha\in S_{m}$, $\beta\in S_{n}$, $\gamma\in S_{k}$ and $\theta=uvw\in S_r$ such that $u,v,w$ are disjoint cycles of $\theta$ given by $u=(i_{1},\dots,i_{s_{1}})$, $v=(i_{s_1+1},\dots,i_{s_{2}})$, $w=(i_{s_2+1},\dots,i_r)$, then the following statements hold.

\begin{enumerate}
\item
$(\alpha\mid\diamond\beta)\mid\diamond\gamma=\alpha\mid\diamond(\beta\mid\diamond\gamma)$ \medskip

\item
$(\alpha\diamond\mid\beta)\diamond\mid\gamma=\alpha\diamond\mid(\beta\diamond\mid\gamma)$
\medskip

\item
$(u\diamond v)\diamond w=u\diamond(v\diamond w)$
\end{enumerate}
	\end{thm}

	\begin{proof} We proceed according to each item.
\begin{enumerate}
\item
Due to $\alpha\in S_{m}, \beta\in S_{n}, \gamma\in S_{k}$, we have
\begin{eqnarray*}
\alpha\mid\diamond\beta & = & \left(\begin{array}{cccccc}
1 & \dots & m & m+1 & \dots & m+n\\
\alpha (1)+n & \dots & \alpha(m)+n & \beta (1) & \dots & \beta (n)\end{array}\right)\\
(\alpha\mid\diamond\beta)\mid\diamond\gamma & = & \Bigg(\begin{array}{cccccc}
1 & \dots & m & m+1 & \dots & m+n\\
\alpha (1)+n+k & \dots & \alpha(m)+n+k & \beta (1)+k & \dots & \beta (n)+k \end{array}\\&&\\
&  & \begin{array}{ccc}
m+n+1 & \dots & m+n+k\\
\gamma(1) & \dots & \gamma (k) \end{array}\Bigg)\\
\end{eqnarray*}
On the other hand
\begin{eqnarray*}
\beta\mid\diamond\gamma & = & \left(\begin{array}{cccccc}
1 & \dots & n & n+1 & \dots & n+k\\
\beta (1)+k & \dots & \beta(n)+k & \gamma (1) & \dots & \gamma (k)\end{array}\right)\\
\alpha\mid\diamond(\beta\mid\diamond\gamma) & = & \Bigg(\begin{array}{cccccc}
1 & \dots & m & m+1 & \dots & m+n\\
\alpha (1)+n+k & \dots & \alpha(m)+n+k & \beta (1)+k & \dots & \beta (n)+k \end{array}\\&&\\
&  & \begin{array}{ccc}
m+n+1 & \dots & m+n+k\\
\gamma(1) & \dots & \gamma (k) \end{array}\Bigg)\\
\end{eqnarray*}
Therefore $(\alpha\mid\diamond\beta)\mid\diamond\gamma=\alpha\mid\diamond(\beta\mid\diamond\gamma)$.
\medskip

\item
Due to $\alpha\in S_{m}, \beta\in S_{n}, \gamma\in S_{k}$, we have
\begin{eqnarray*}
\alpha\diamond\mid \beta & = & \left(\begin{array}{cccccc}
1 & \dots & m & m+1 & \dots & m+n\\
\alpha (1) & \dots & \alpha(m) & \beta (1)+m & \dots & \beta (n)+m\end{array}\right)\\
(\alpha\diamond\mid\beta)\diamond\mid\gamma & = & \Bigg(\begin{array}{cccccc}
1 & \dots & m & m+1 & \dots & m+n\\
\alpha (1) & \dots & \alpha(m)& \beta (1)+m & \dots & \beta (n)+m \end{array}\\&&\\
&  & \begin{array}{ccc}
m+n+1 & \dots & m+n+k\\
\gamma(1)+m+n & \dots & \gamma (k)+m+n \end{array}\Bigg)\\
\end{eqnarray*}
On the other hand
\begin{eqnarray*}
\beta\diamond\mid \gamma & = & \left(\begin{array}{cccccc}
1 & \dots & n & n+1 & \dots & n+k\\
\beta (1) & \dots & \beta(n) & \gamma (1)+n & \dots & \gamma (k)+n\end{array}\right)\\
\alpha\diamond\mid (\beta\diamond\mid \gamma) & = & \Bigg(\begin{array}{cccccc}
1 & \dots & m & m+1 & \dots & m+n\\
\alpha (1) & \dots & \alpha(m) & \beta (1)+m & \dots & \beta (n)+m \end{array}\\&&\\
&  & \begin{array}{ccc}
m+n+1 & \dots & m+n+k\\
\gamma(1)+m+n & \dots & \gamma (k)+m+n \end{array}\Bigg)\\
\end{eqnarray*}
Therefore $(\alpha\diamond\mid\beta)\diamond\mid\gamma=\alpha\diamond\mid(\beta\diamond\mid\gamma)$.
\medskip

\item
Due to $u=(i_{1},\dots,i_{s_{1}})$, $v=(i_{s_1+1},\dots,i_{s_{2}})$ and  $w=(i_{s_2+1},\dots,i_r)$, we have
\begin{eqnarray*}
u\diamond v & = & (i_{1},\dots,i_{s_{1}},i_{s_{1}+1},\dots,i_{s_{2}})\\
(u \diamond v)\diamond w & = & (i_{1},\dots,i_{s_{1}},i_{s_{1}+1},\dots,i_{s_{2}},i_{s_{1}+s_{2}+1},\dots,i_r)\\
\end{eqnarray*}
On the other hand
\begin{eqnarray*}
v \diamond w & = & (i_{s_1+1},\dots,i_{s_{2}},i_{s_{2}+1},\dots,i_{r})\\
u\diamond (v\diamond w) & = & (i_{1},\dots,i_{s_{1}},i_{s_{1}+1},\dots,i_{s_{2}},i_{s_{2}+1},\dots,i_r)\\
\end{eqnarray*}
Therefore $(u\diamond v)\diamond w=u\diamond(v\diamond w)$.
\end{enumerate}
Thus, the proof is complete.
	\end{proof}
The following result corresponds to a sort of non-commutative De Morgan's laws for Pasting and Reversing of Permutations. We will not consider an analogous for disjoint cycles of permutations.
	\begin{thm}[De Morgan's laws for Pasting and Reversing of Permutations]\label{th6}
Consider  $\alpha\in S_{m}, \beta\in S_{n}$, then the following statements hold.
\begin{enumerate}
\item
$\widetilde{\alpha\mid\diamond\beta}=\widetilde{\beta}\diamond\mid\widetilde{\alpha}$ \medskip

\item$\widetilde{\alpha\diamond\mid\beta}=\widetilde{\beta}\mid\diamond\widetilde{\alpha}$
\end{enumerate}
	\end{thm}
\begin{proof}

Due to $\alpha\in S_{m}, \beta\in S_{n},$ we have that
\begin{eqnarray*}
\widetilde{\alpha\mid\diamond\beta} & = & \widetilde{\left(\begin{array}{cccccc}
1 & \dots & m & m+1 & \dots & m+n\\
\alpha (1)+n & \dots & \alpha (m)+n & \beta (1) & \dots & \beta (n)\end{array}\right)}\\
 & = & \left(\begin{array}{cccccccc}
1 & \dots & n & n+1 & \dots & m & \dots & m+n\\
\beta (n) & \dots & \beta (1) & \alpha (m)+n & \dots & \alpha (1) & \dots & \alpha (1)+n\end{array}\right)\\
 & = & \left(\begin{array}{ccc}
1 & \dots & n\\
\beta (n) & \dots & \beta (1)\end{array}\right)\diamond\mid\left(\begin{array}{ccc}
1 & \dots & m\\
\alpha (m) & \dots & \alpha (1)\end{array}\right)\\
 & = & \widetilde{\left(\begin{array}{ccc}
1 & \dots & n\\
\beta (1) & \dots & \beta (n)\end{array}\right)}\diamond\mid\widetilde{\left(\begin{array}{ccc}
1 & \dots & m\\
\alpha (1) & \dots & \alpha (m)\end{array}\right)}\\
 & = & \widetilde{\beta}\diamond\mid\widetilde{\alpha}\\
\end{eqnarray*}\medskip
In a similar way,
\begin{eqnarray*}
\widetilde{\alpha\diamond\mid\beta} & = & \widetilde{\left(\begin{array}{cccccc}
1 & \dots & m & m+1 & \dots & m+n\\
\alpha (1) & \dots & \alpha (m) & \beta (1)+m & \dots & \beta (n)+m\end{array}\right)}\\
 & = & \left(\begin{array}{cccccc}
1 & \dots & n & n+1 & \dots & m+n\\
\beta (n)+m & \dots & \beta (1)+m & \alpha (m) & \dots & \alpha (1)\end{array}\right)\\
 & = & \left(\begin{array}{ccc}
1 & \dots & n\\
\beta (n) & \dots & \beta (1)\end{array}\right)\mid\diamond\left(\begin{array}{ccc}
1 & \dots & m\\
\alpha (m) & \dots & \alpha (1)\end{array}\right)\\
 & = & \widetilde{\left(\begin{array}{ccc}
1 & \dots & n\\
\beta (1) & \dots & \beta (n)\end{array}\right)}\mid\diamond\widetilde{\left(\begin{array}{ccc}
1 & \dots & m\\
\alpha (1) & \dots & \alpha (m)\end{array}\right)}\\
 & = & \widetilde{\beta}\mid\diamond\widetilde{\alpha}\\
\end{eqnarray*}
Completing the proof.
\end{proof}

\section{Genealogy of Simple Permutations}

To start a genealogy, we need predecessors and successors of a member of the family.  Is natural to see that any element in a genealogy has only one immediate predecessor, but not necessarily it has one immediate successor, as for example any people can have only one father but more than one son. Now, the genealogy of a member of the family is the tree that begin with the first predecessor and include such member with their successors. When the time tends to infinity, the number of predecessors of such member is finite, while its number of successors is not finite. These rules can apply for simple permutations too, up to special cases. The following definitions are in agreement with the ideas presented in \cite{Genealogia,Bernhardt}.
\begin{defn}[Predecessors and Successors of Simple Permutations]\label{def12}
Given three simple permutations $\theta_k\in S_k$, $\theta_m\in S_m$ and $\theta_n\in S_n$, we say that $\theta_n$ is an immediate predecessor of $\theta_m$ (denoted by $\theta_n\prec \theta_m$) and $\theta_k$ is an immediate successor  of $\theta_m$ (denoted by $\theta_m\prec \theta_k$) whether they satisfy each one of the following conditions:
\begin{enumerate}
\item $k<m<n$,
\item  $\theta_k\lhd\theta_m\lhd \theta_n$ (resp.  $\theta_n\lhd\theta_m\lhd \theta_k$),
\item  $k\lhd m\lhd n$ (resp. $n\lhd m\lhd k$),
\item $k,m,n$ (resp. $n,m,k$) are consecutive integers in the Sharkovskii order.
\end{enumerate}
The first predecessor of $\theta_m$ is the simple permutation that has not immediate successor in a chain of immediate predecessors that ends with $\theta_m$. A $j$-th successor (resp. predecessor) of $\theta_m$ is the simple permutation placed in the position $j+1$ (resp. $1$) in a chain containing $j$ immediate successors (resp. predecessors) of $\theta_m$ starting (resp. ending) with itself. The simple permutation $\theta$ is a predecessor (resp. successor) of $\theta_m$ whether there is a finite chain of immediate predecessors (resp. succesors) between $\theta$ and $\theta_m$ (resp. $\theta_m$ and $\theta$).
\end{defn}
According to Definition \ref{def12}, if $\theta$ is of odd order (resp. power of two), their predecessors an successors also will be of odd order (resp. power of two). Therefore, predecessors and successors for simple permutations with mixed order depend on the required block on the middle tail in the Sharkovskii ordering. That is, for $s$ fixed, we can get predecessors and successors of a simple permutation with order $2^s(2k+1)$. In \cite{Genealogia} was introduced the concept \emph{Genealogy} to describe the tree of predecessors and successors of simple permutations with order a power of two. In general, we have the following definition.

\begin{defn}[Genealogy of Simple Permutations]\label{def13}
Let $\theta$ be a simple permutation. The tree formed by all predecessors and all successors of $\theta$ is called the genealogy of $\theta$. A branch of the genealogy of $\theta$ is a chain which includes predecessors and successors of $\theta$, but starting with the first predecessor of $\theta$ and any simple permutation has only one immediate successor. 
\end{defn}
We illustrate the previous definitions with the following example.
\begin{ex}[Stefan Orbits] The genealogy of permutations with odd order is 
$$\left\lbrace\begin{array}{c}
\alpha_3\prec \alpha_5\prec \cdots \prec  \alpha_{2n-1}\prec \cdots\\
 \beta_3\prec \beta_5\prec \cdots\prec \beta_{2n-1}\prec \cdots
\end{array}\right.
 $$
 This genealogy, which has not the first predecessor, has two isolated branches $$\alpha_3\prec \alpha_5\prec\cdots \prec \alpha_{2n-1}\prec\cdots\quad \textrm{and}\quad \beta_3\prec \beta_5\prec\cdots\prec \beta_{2n-1}\prec\cdots$$
 with first predecessors $\alpha_3$ and $\beta_3$ respectively. We see that immediate successors of elements belonging to $Sim(2k-1)$ are elements belonging to $Sim(2k+1)$ and immediate predecessors of elements belonging to $Sim(2k-1)$ are elements of $Sim(2k-3)$.
Finally, we observe that $\theta\prec \phi$ implies that $\theta \lhd \phi$, where $\theta$ and $\phi$ are simple permutations with odd order.
\end{ex}

The next aim is to illustrate the genealogy for simple permutations with order a power of two, which has been described by Bernhardt \cite{Bernhardt} through transpositions and by Acosta-Hum\'anez \cite{Genealogia} through Pasting and Reversing. The following has been adapted from \cite{Bernhardt}.
	\begin{defn}\label{def14}
If $\theta\in Sim(2^{n})$ then $\theta^{*}\in S_{2^{n+1}}$ and  $\theta_{*}\in S_{2^{n-1}}$, are defined by
\begin{eqnarray*}
\theta^{*}(2k)&=&2\theta{(k)}, \, \,\,\theta^{*}(2k-1)=2\theta(k)-1,\\
\theta_{*}(k)&=&\left\lfloor\frac{1}{2}(\theta{(2k)}+1)\right\rfloor\\
\end{eqnarray*}
where $\lfloor x \rfloor$ is the greatest integer less than or equal to $x$.
	\end{defn}

We consider the transposition $\rho_{s}$ and the permutation $e_n$ given by
\begin{eqnarray*}
\rho_s:=\left(\begin{array}{cc}
2s-1 & 2s\\
2s & 2s-1\end{array}\right),\,\,\,
e_n  :=  \left(\begin{array}{cccccc}
1 & 2 & 3 & \dots &n-1 &n\\
1 & 2 & 3 & \dots &n-1 &n\\
\end{array}\right).\\
\end{eqnarray*}
  By $\rho_k\circ\rho_m$  we denote the composition of transpositions $\rho_k$ and $\rho_s$.
The following theorem, see \cite{Bernhardt}, states the existence of predecessors and successors of a given simple permutation with order a power of two and a way to obtain them. 

	\begin{thm}[Bernhardt, \cite{Bernhardt}]\label{th7}
Consider  $\eta \in Sim(2^{n+1})$, $\theta\in Sim(2^{n})$, $\phi\in Sim(2^{n-1})$.
If $\eta \triangleleft \theta \triangleleft \phi$, then $\phi\prec \theta\prec \eta$, $\eta=\theta^{*}\circ\rho_{i_{1}}\circ\dots\circ\rho_{i_{2m-1}}$ and $\phi=\theta_{*}$.
	\end{thm}

The following theorem, due to the first author in \cite{Genealogia}, shows a way to construct the same genealogy by using Pasting and Reversing. It is an important result because it is the first use of Pasting and Reversing in combinatorial dynamics.

\begin{thm}[Acosta-Hum\'anez, \cite{Genealogia}]\label{th8}
Consider $n=2^{k+1}$, $k\in\mathbb{Z}$, $\theta_{1}=\phi_{1}=(1)$, $\theta_{n}$ and $\phi_{n}$ permutations as follows:
\begin{equation}\label{eq1}
\theta_{n}  =  e_{\frac{n}{2}}\mid\diamond\theta_{\frac{n}{2}},
\end{equation}
\begin{equation}\label{eq2}
\phi_{n}  =\widetilde{e_{\frac{n}{2}}}\mid\diamond\widetilde{\phi_{\frac{n}{2}}}.
\end{equation}

The following statements hold.
\begin{enumerate}
\item  $\phi_n=\widetilde{\phi_{\frac{n}{2}}\diamond\mid e_{\frac{n}{2}}},$
\item
$\theta_{n},\phi_{n}\in Sim(n),$
\item
$\theta_{\frac{n}{2}}\prec\theta_{n}\prec\theta_{2n},$
\item $\phi_{\frac{n}{2}}\prec\phi_{n}\prec\phi_{2n},$

\item
$(\theta_{2n})_{*}=\theta_{n}=(\theta_{\frac{n}{2}})^{*}\circ\rho_{\frac{n}{2}}$,
\item  $(\phi_{2n})_{*}=\phi_{n}=(\phi_{\frac{n}{2}})^{*}\circ\rho_{1}\dots\rho_{\frac{n-2}{2}}.$
\item Let $u$ and $v$ be disjoint cycles of $(\theta_n)^*$. The permutation $(\theta_n)^*=uv$ if and only if $(\theta_{2n})=u\diamond v$.
\item Let $p$ and $q$ be disjoint cycles of $(\phi_n)^*$. The permutation $(\phi_n)^*=pq$ if and only if exists $\eta_{2n}\in Sim(2n)$, being $\eta_{2n}\neq \phi_{2n}$, such that $\eta_{2n}=p\diamond q$.
\item $\varphi=(\theta_{n}\circ\rho_{k})\circ(\phi_{n}\circ\rho_{j}),\forall\varphi\in Sim(n)$, where $\rho_{k}$ and $\rho_{j}$ are compositions of odd length transpositions,
\item $\varphi^2=\varphi_*\diamond\mid \varphi_*,\forall\varphi\in Sim(n)$.
\end{enumerate}
	\end{thm}

According to Theorem \ref{th8} we can construct in an explicit way two branches of the genealogy of simple permutations with order a power of two, which has the first predecessor $(1)=\phi_1=\theta_1$. This means that the genealogy starts with $(1)$ and any branch of the genealogy must include it as the first predecessor. Furthermore, we see that if $\phi\prec\theta\prec \eta$, then $\eta\lhd \theta\lhd\phi$, which is different to the case of simple permutations with odd order, such as we see in the following example.

\begin{ex}[Branches $\theta_{2^k}$ and $\phi_{2^k}$]
The branch $\theta_{2^k}$ writing as cycles is
\begin{eqnarray*}
(1)\prec (1,2)\prec (1,3,2,4)\prec (1,5,3,7,2,6,4,8)\prec\\
(1,9,5,13,3,11,7,15,2,10,6,14,4,12,8,16)\prec \cdots
\end{eqnarray*}
The branch $\phi_{2^k}$ writing as cycles is
\begin{eqnarray*}
(1)\prec (1,2)\prec (1,4,2,3)\prec (1,8,4,5,2,7,3,6)\prec\\
(1,16,8,9,4,13,5,12,2,15,7,10,3,14,6,11)\prec \cdots
\end{eqnarray*}
\end{ex}

Theorem \ref{th8} and the genealogy of simple permutations with odd order inspired the next section, genealogy of simple permutations with mixed order $4n+2$ (mixed order $(2n+1)2^s$ with $s=1$ and $n>0$), which contains the main results and a lot of examples of this paper.
\section{Main Results}
According to the previous section, immediate successors of elements belonging to $Sim(4n-2)$ are elements of $Sim(4n+2)$ and immediate predecessors of elements belonging to $Sim(4n+2)$ are elements of $Sim(4n-2)$. In this section we determine some branches of the genealogy of simple permutations with mixed order $4n+2$ as well the complete the explicit forms of $Sim(6)$ and $Sim(10)$.
\subsection{Genealogy in $Sim(4n+2)$}
 The following theorems correspond to our main results.

\begin{thm}\label{th9}
If $\theta\in Sim(4n+2)$, its genealogy has only four branches and the number of members of its genealogy until itself is $8n+4$.
\end{thm}
\begin{proof}
As $\theta$ is simple, $\theta\left[P\left(4n+2,2,j\right)\right] = P\left(4n+2,2,\sigma\left(j\right)\right)$, where
$\sigma=(1,2)$, it implies that the first element of those simple permutations has to be greater than ${2n+1}$. Now, $\theta^2$ restricted to $P\left(4n+2,2,j\right)$ is simple for all $j=1,2$. Therefore, $\theta(1)=k$, with ${k=2n+2,\ldots,4n+2}$. For instance, there exist $4$ simple permutations with order $2n+1$ in the restriction of $\theta^2$, that is, the genealogy of $\theta$ has four branches. On the other hand, the first element of a simple permutation in this order has to be taken from $\lbrace2n+2, 2n+3, \dots, 4n+2\rbrace$ in order to accomplish the conditions 3.a and 3.b in Definition \ref{def6}. Thus, there are $2n+1$ ways to obtain a simple permutation associated to each one of the 4 possible $\theta^{2}$. It means, $8n+4$ simple permutations with order $4n+2$, which corresponds to the number of members of its genealogy until itself.
\end{proof}

\begin{thm}\label{th10}
Consider the two branches of the genealogy of simple permutations with odd order $\alpha_{2n+1}$ and $\beta_{2n+1}$ given in Definition \ref{def6}. Consider the set $$\mathfrak{P}:=\{\theta_{4n+2}, \phi_{4n+2},\eta_{4n+2},\varphi_{4n+2}\}\subset S_{4n+2}$$ where the permutations are defined as follows:
\begin{equation}\label{eq3}
\theta_{4n+2} : =\alpha_{2n+1}\mid \diamond e_{2n+1},
\end{equation}
\begin{equation}\label{eq4}
\eta_{4n+2} : =e_{2n+1}\mid \diamond\alpha_{2n+1},
\end{equation}
\begin{equation}\label{eq5}
\phi_{4n+2} : =\beta_{2n+1}\mid  \diamond e_{2n+1},
\end{equation}
\begin{equation}\label{eq6}
\varphi_{4n+2} : =e_{2n+1}\mid \diamond\beta_{2n+1}.
\end{equation}
The following statements hold.
\begin{enumerate}
\item $\theta^2_{4n+2}=\eta^2_{4n+2}=\alpha_{2n+1} \diamond\mid \alpha_{2n+1},$  $\phi^2_{4n+2}=\varphi^2_{4n+2}=\beta_{2n+1} \diamond\mid \beta_{2n+1},$
\item $\mathfrak{P}\subset Sim(4n+2)$,  
\item $\theta_{4n-2}\prec\theta_{4n+2}\prec\theta_{4n+6},$ $\phi_{4n-2}\prec\phi_{4n+2}\prec\phi_{4n+6},$
 $\eta_{4n-2}\prec\eta_{4n+2}\prec\eta_{4n+6},$
 $\varphi_{4n-2}\prec\varphi_{4n+2}\prec\varphi_{4n+6}.$
\end{enumerate}
\end{thm}
\begin{proof} We proceed according to each item.
\begin{enumerate}
\item  Due to $\theta_{4n+2}=\alpha_{2n+1}\mid\diamond e_{2n+1}\in S_{4n+2}$ and $\eta_{4n+2}  =e_{2n+1}\mid \diamond\alpha_{2n+1}\in S_{4n+2}$, we have that
$$\theta_{4n+2}=\left(\begin{array}{cccccc}
1  & \dots & 2n+1 & 2n+2 &  \dots & 4n+2\\
\alpha_{2n+1}(1)+2n+1  & \dots &\alpha_{2n+1}(2n+1)+2n+1 & 1 &  \ldots & 2n+1
\end{array}\right)$$ and
$$\eta_{4n+2}=\left(\begin{array}{cccccc}
1  & \dots & 2n+1 & 2n+2 &  \dots & 4n+2\\
 2n+2 &  \ldots & 4n+2&\alpha_{2n+1}(1)  & \dots &\alpha_{2n+1}(2n+1) \end{array}\right).$$ Owing  to $\alpha_{2n+1}= (1,2n+1,n+1,n,n+2,n-1,n+3,\dots,2,2n)$, therefore 
$$\theta_{4n+2}=\left(\begin{array}{ccccccccc}
1  &2& \dots&2n & 2n+1 & 2n+2 &  \dots&4n+1 & 4n+2\\
4n+2 &4n+1 &\dots&2n+2 &3n+2&1&\ldots&2n &2n+1
\end{array}\right)$$ and
$$\eta_{4n+2}=\left(\begin{array}{ccccccccc}
1  &2& \dots & 2n+1 & 2n+2 &2n+3&  \dots&4n+1 & 4n+2\\
2n+2&2n+3&\ldots &4n+2&2n+1 &2n &\dots & 1&n+1
\end{array}\right).$$
 We see that $\theta^2_{2n+1}=\eta^2_{2n+1}$. Moreover,

\begin{eqnarray*}
\theta_{4n+2}^{2} &
= & \left(\begin{array}{cccccccccc}
1 & 2 & \dots &2n& 2n+1 & 2n+2 &2n+3& \dots &4n+1 & 4n+2\\
2n+1 &2n & \dots & 1& n+1&4n+2&4n+1& \dots &2n+2 & 3n+2\end{array}\right)\\
& = & \left(\begin{array}{ccc}
1  & \dots & 2n+1\\
{\alpha_{2n+1}({1})}  & \dots & {\alpha_{2n+1}(2n+1)}\end{array}\right)
\diamond\mid
\left(\begin{array}{ccccccc}
1  & \dots & 2n+1\\
{\alpha_{2n+1}({1})}  & \dots & {\alpha_{2n+1}(2n+1)}\end{array}\right)\\
& = & \alpha_{2n+1}\diamond\mid\alpha_{2n+1}\end{eqnarray*}
Thus, we proved that $\theta^2_{2n+1}=\eta^2_{2n+1}=\alpha_{2n+1}\diamond\mid\alpha_{2n+1}$.
Now, assuming $r=2n+1$ and due to $\phi_{2r}=\beta_{r}\mid\diamond e_{r}\in S_{2r}$, $\varphi_{2r}  =e_{r}\mid \diamond\beta_r\in S_{2r}$, we have that
$$\phi_{2r}=\left(\begin{array}{cccccc}
1  & \dots & r & r+1 &  \dots & 2r\\
\beta_{r}(1)+r  & \dots &\beta_{r}(r)+r & 1 &  \ldots & r\end{array}\right)$$ and
$$\varphi_{2r}=\left(\begin{array}{cccccc}
1  & \dots & r & r+1 &  \dots & 2r\\
 r+1 &  \ldots & 2r&\beta_{r}(1)  & \dots &\beta_{r}(r) \end{array}\right).$$ Owing to $\beta_{r}=(1,\frac{r+1}2,\frac{r+3}2,\frac{r-1}2,\frac{r+5}2,\frac{r-3}2,\dots,2,r)$, we obtain
 $$\phi_{2r}=\left(\begin{array}{cccccc}
1  & \dots & r & r+1 &  \dots & 2r\\
\frac{3r+1}{2} & \dots &r+1 & 1 &  \ldots & r\end{array}\right)$$ and
$$\varphi_{2r}=\left(\begin{array}{cccccc}
1  & \dots & r & r+1 &  \dots & 2r\\
 r+1 &  \ldots & 2r&\frac{r+1}{2}  & \dots &1 \end{array}\right).$$ 
   and proceeding as for the previous case, we have

\begin{eqnarray*}
\phi_{2r}^{2} &
= & \left(\begin{array}{cccccccc}
1 & 2 & \dots & r & r+1 &r+2 & \dots & 2r\\
{\beta_r({1})} & {\beta_r({2})} & \dots & {\beta_r{(r)}} & {\beta_r{(1)}+r} & {\beta_r{(2)}+r} & \dots & {\beta_r{(r)}+r}\end{array}\right)\\
& = & \left(\begin{array}{cccc}
1 & 2 & \dots & n\\
{\beta_r({1})} & {\beta_r({2})} & \dots & {\beta_r(r)}\end{array}\right)
\diamond\mid
\left(\begin{array}{cccccccc}
1 & 2 & \dots & r\\
{\beta_r({1})} & {\beta_r({2})} & \dots & {\beta_{r}(r)}\end{array}\right)\\
& = & \beta_r\diamond\mid\beta_r\end{eqnarray*}
and 
\begin{eqnarray*}
\varphi_{2r}^{2} &
= & \left(\begin{array}{cccccccc}
1 & 2 & \dots & r & r+1 &r+2 & \dots & 2r\\
{\beta_r({1})} & {\beta_r({2})} & \dots & {\beta_r{(r)}} & {\beta_r{(1)}+r} & {\beta_r{(2)}+r} & \dots & {\beta_r{(r)}+r}\end{array}\right)\\
& = & \phi_{2r}^{2}\\
& = & \beta_r\diamond\mid\beta_r\end{eqnarray*}

\item By previous item we have that $\theta_{4n+2}$, $\eta_{4n+2}$, $\phi_{4n+2}$ and $\varphi_{4n+2}$ are $(4n+2)$-cycles, thus, $\mathfrak{P}\subset\in C_{4n+2}$.  Now, we will see that $\theta_{4n+2}$, $\eta_{4n+2}$, $\phi_{4n+2}$ and $\varphi_{4n+2}$ satisfy 3 in Definition \ref{def6}. We start looking condition 3.a, thus we have that
\begin{eqnarray*}
\theta_{4n+2}\left[P\left(4n+2,2,1\right)\right] & = & \theta_{4n+2}\{1, 2, \ldots, 2n,2n+1 \} \\
& = & \{\alpha_{2n+1}(1)+2n+1,\ldots,{\alpha_{2n+1}}(2n+1)+2n+1\}\\
& = & \left\{ 2n+2,2n+3,\ldots, 4n+1,4n+2\right\}\\
& = & P\left(10,2,2\right)\\
& = & P\left(10,2,\sigma\left(1\right)\right)
\end{eqnarray*}
\begin{eqnarray*}
\theta_{4n+2}\left[P\left(4n+2,2,2\right)\right] & = & \theta_{4n+2}\{2n+1, 2n+2, \ldots, 4n+1,4n+2 \} \\
& = & \{e_{2n+1}(1),e_{2n+1}(2),\ldots,{e_{2n+1}}(2n),{e_{2n+1}}(2n+1)\}\\
& = & \left\{ 1,2,\ldots, 2n,2n+1\right\}\\
& = & P\left(10,2,1\right)\\
& = & P\left(10,2,\sigma\left(2\right)\right)
\end{eqnarray*}

Now, we see whether $\theta_{4n+2}$ satisfies the condition 3.b in Definition \ref{def6}. By the previous item we have that,
\begin{eqnarray*}
\theta^2_{4n+2}\left[P\left(4n+2,2,1\right)\right] & = & \alpha_{2n+1}\{1, 2, \ldots, 2n,2n+1 \} \\
& = & \{\alpha_{2n+1}(1),\ldots,{\alpha_{2n+1}}(2n+1)\}\\
& = & \left\{ 1, 2,\ldots, 2n,2n+1\right\}\\
& = & P\left(4n+2,2,1\right)
\end{eqnarray*}
\begin{eqnarray*}
\theta^2_{4n+2}\left[P\left(4n+2,2,2\right)\right] & = & \theta^2_{4n+2}\{2n+2, 2n+3, \ldots, 4n+1,4n+2 \} \\
& = & \{\alpha_{2n+1}(1)+2n+1,\ldots,{\alpha_{2n+1}}(2n+1)+2n+1\}\\
& = & \left\{ 2n+2, 2n+3,\ldots, 4n+1,4n+2\right\}\\
& = & P\left(4n+2,2,2\right)
\end{eqnarray*}
For instance, $\theta_{4n+2}^2$ restricted to $P\left(4n+2,2,j\right)$, being $j=1,2$, is $\alpha_{2n+1}$, which is simple. Thus, we proved that $\theta_{4n+2}\in Sim(4n+2)$.

Similarly for $\eta_{4n+2}$, we start condition 3.a, thus we have that
\begin{eqnarray*}
\eta_{4n+2}\left[P\left(4n+2,2,1\right)\right] & = & \eta_{4n+2}\{1, 2, \ldots, 2n,2n+1 \} \\
& = & \{\alpha_{2n+1}(1)+2n+1,\ldots,{\alpha_{2n+1}}(2n+1)+2n+1\}\\
& = & \left\{ 2n+2,2n+3,\ldots, 4n+1,4n+2\right\}\\
& = & P\left(10,2,2\right)\\
& = & P\left(10,2,\sigma\left(1\right)\right)
\end{eqnarray*}
\begin{eqnarray*}
\eta_{4n+2}\left[P\left(4n+2,2,2\right)\right] & = & \eta_{4n+2}\{2n+1, 2n+2, \ldots, 4n+1,4n+2 \} \\
& = & \{e_{2n+1}(1),e_{2n+1}(2),\ldots,{e_{2n+1}}(2n),{e_{2n+1}}(2n+1)\}\\
& = & \left\{ 1,2,\ldots, 2n,2n+1\right\}\\
& = & P\left(10,2,1\right)\\
& = & P\left(10,2,\sigma\left(2\right)\right)
\end{eqnarray*}

Now, we see that $\eta_{4n+2}$ satisfies the condition 3.b in Definition \ref{def6} because, by the previous item, $\eta_{4n+2}^2=\theta_{4n+2}^2$. Thus, we proved that $\eta_{4n+2}\in Sim(4n+2)$.

The following step is the analysis of $\phi_{4n+2}$. We start looking condition 3.a in Definition \ref{def6}, thus we have that
\begin{eqnarray*}
\phi_{4n+2}\left[P\left(4n+2,2,1\right)\right] & = & \phi_{4n+2}\{1, 2, \ldots, 2n,2n+1 \} \\
& = & \{\beta_{2n+1}(1)+2n+1,\ldots,{\beta_{2n+1}}(2n+1)+2n+1\}\\
& = & \left\{ 2n+2,2n+3,\ldots, 4n+1,4n+2\right\}\\
& = & P\left(10,2,2\right)\\
& = & P\left(10,2,\sigma\left(1\right)\right)
\end{eqnarray*}
\begin{eqnarray*}
\phi_{4n+2}\left[P\left(4n+2,2,2\right)\right] & = & \phi_{4n+2}\{2n+1, 2n+2, \ldots, 4n+1,4n+2 \} \\
& = & \{e_{2n+1}(1),e_{2n+1}(2),\ldots,{e_{2n+1}}(2n),{e_{2n+1}}(2n+1)\}\\
& = & \left\{ 1,2,\ldots, 2n,2n+1\right\}\\
& = & P\left(10,2,1\right)\\
& = & P\left(10,2,\sigma\left(2\right)\right)
\end{eqnarray*}

Now, we see whether $\phi_{4n+2}$ satisfies the condition 3.b in Definition \ref{def6}. By the previous item we have that,
\begin{eqnarray*}
\phi^2_{4n+2}\left[P\left(4n+2,2,1\right)\right] & = & \beta_{2n+1}\{1, 2, \ldots, 2n,2n+1 \} \\
& = & \{\beta_{2n+1}(1),\ldots,{\beta_{2n+1}}(2n+1)\}\\
& = & \left\{ 1, 2,\ldots, 2n,2n+1\right\}\\
& = & P\left(4n+2,2,1\right)
\end{eqnarray*}
\begin{eqnarray*}
\phi^2_{4n+2}\left[P\left(4n+2,2,2\right)\right] & = & \phi^2_{4n+2}\{2n+2, 2n+3, \ldots, 4n+1,4n+2 \} \\
& = & \{\beta_{2n+1}(1)+2n+1,\ldots,{\beta_{2n+1}}(2n+1)+2n+1\}\\
& = & \left\{ 2n+2, 2n+3,\ldots, 4n+1,4n+2\right\}\\
& = & P\left(4n+2,2,2\right)
\end{eqnarray*}
For instance, the restriction of $\phi_{4n+2}^2$ to $P\left(4n+2,2,j\right)$, being $j=1,2$, is $\beta_{2n+1}$, which is simple. Thus, we proved that $\phi_{4n+2}\in Sim(4n+2)$.

Finally, we analyse to $\varphi_{4n+2}$. We start looking condition 3.a, thus we have that
\begin{eqnarray*}
\varphi_{4n+2}\left[P\left(4n+2,2,1\right)\right] & = & \varphi_{4n+2}\{1, 2, \ldots, 2n,2n+1 \} \\
& = & \{\beta_{2n+1}(1)+2n+1,\ldots,{\beta_{2n+1}}(2n+1)+2n+1\}\\
& = & \left\{ 2n+2,2n+3,\ldots, 4n+1,4n+2\right\}\\
& = & P\left(10,2,2\right)\\
& = & P\left(10,2,\sigma\left(1\right)\right)
\end{eqnarray*}
\begin{eqnarray*}
\varphi_{4n+2}\left[P\left(4n+2,2,2\right)\right] & = & \varphi_{4n+2}\{2n+1, 2n+2, \ldots, 4n+1,4n+2 \} \\
& = & \{e_{2n+1}(1),e_{2n+1}(2),\ldots,{e_{2n+1}}(2n),{e_{2n+1}}(2n+1)\}\\
& = & \left\{ 1,2,\ldots, 2n,2n+1\right\}\\
& = & P\left(10,2,1\right)\\
& = & P\left(10,2,\sigma\left(2\right)\right)
\end{eqnarray*}

Now, we see that $\varphi_{4n+2}$ satisfies the condition 3.b in Definition \ref{def6} because, by the previous item, $\varphi_{4n+2}^2=\phi_{4n+2}^2$. Thus, we proved that $\varphi_{4n+2}\in Sim(4n+2)$. In conclusion, $\mathfrak{P}\subset Sim(4n+2)$.
\item By Sharkovskii's order we have that \begin{eqnarray*}
{4n-2}\lhd\theta_{4n+2}\lhd\theta_{4n+6},\\
\phi_{4n-2}\lhd\phi_{4n+2}\lhd\phi_{4n+6},\\
\eta_{4n-2}\lhd\eta_{4n+2}\lhd\eta_{4n+6},\\
\varphi_{4n-2}\lhd\varphi_{4n+2}\lhd\varphi_{4n+6}.
\end{eqnarray*} On the other hand, due to $\mathfrak{P}\subset Sim(4n+2)$, $4n-2< 4n+2<4n+6$  and there is not intermediate elements in $4n-2\lhd 4n+2\lhd 4n+6$, we obtain  $\theta_{4n-2}\prec\theta_{4n+2}\prec\theta_{4n+6},$ $\phi_{4n-2}\prec\phi_{4n+2}\prec\phi_{4n+6},$
 $\eta_{4n-2}\prec\eta_{4n+2}\prec\eta_{4n+6},$
 $\varphi_{4n-2}\prec\varphi_{4n+2}\prec\varphi_{4n+6}.$

\end{enumerate}
\end{proof}

To illustrate the previous theorems, we introduce the following examples.
\begin{ex}
The permutations $$\theta_6=\alpha_{3}\mid\diamond e_{3}=\left(\begin{array}{cccccc}
1 & 2 & 3 & 4 & 5 & 6\\
6 & 4 & 5 & 1 & 2 & 3\end{array}\right),$$ $$\phi_6=\beta_{3}\mid\diamond e_{3}=\left(\begin{array}{cccccc}
1 & 2 & 3 & 4 & 5 & 6\\
5 & 6 & 4 & 1 & 2 & 3\end{array}\right),$$ $$\eta=e_{3}\mid\diamond \alpha_{3}=\left(\begin{array}{cccccc}
1 & 2 & 3 & 4 & 5 & 6\\
4 & 5 & 6 & 3 & 1 & 2\end{array}\right)$$ and $$\varphi=e_{3}\mid\diamond \beta_{3}=\left(\begin{array}{cccccc}
1 & 2 & 3 & 4 & 5 & 6\\
4 & 5 & 6 & 2 & 3 & 1\end{array}\right)$$ are simple permutations of order 6.
\end{ex}
\begin{ex}
The branches $\theta_{4n+2}$, $\phi_{4n+2}$, $\eta_{4n+2}$ and $\varphi_{4n+2}$ of the genealogy associated to $Sim(4n+2)$ are respectively:
$$(1,6,3,5,2,4)\prec (1,10,5,8,3,7,2,9,4,6)\prec\cdots$$
$$(1,5,2,6,3,4)\prec (1,8,3,9,4,7,2,10,5,6)\prec \cdots$$
 $$(1,4,3,6,2,5)\prec(1,6,5,10,3,8,2,7,4,9)\prec \cdots$$
 $$(1,4,2,5,3,6)\prec(1,6,3,8,4,9,2,7,5,10)\prec\cdots$$
\end{ex}

\subsection{Permutations belonging to $Sim(6)$ and $Sim(10)$}
Here we use Pasting operation and simple permutations of odd order to construct
 simple permutations with order $6$ and $10$. Definition \ref{def6} allows to suggest a way to construct simple permutations $\theta$ taking as reference $\theta^{2}$ and its relation with $\alpha$ and/or $\beta$.  We illustrate this in the following example.
\begin{ex}
Consider $\alpha_3$, $\beta_3$ and
 $\sigma=(1,2)$. Let $\theta$ be the permutation $$\theta=\left(\begin{array}{cccccc}
1 & 2 & 3 & 4 & 5 & 6\\
6 & 5 & 4 & 1 & 3 & 2\end{array}\right).$$
We see that
 \begin{eqnarray*}
\theta\left[P\left(6,2,1\right)\right] & = & \theta\left[\left\{ 1,2,3\right\} \right]=\left\{ 6,5,4\right\} =\left\{ 4,5,6\right\} =P\left(6,2,\sigma\left(1\right)\right)=P\left(6,2,2\right)\end{eqnarray*}
\begin{eqnarray*}
\theta\left[P\left(6,2,2\right)\right] & = & \theta\left[\left\{ 4,5,6\right\} \right]=\left\{ 1,3,2\right\} =\left\{ 1,2,3\right\} =P\left(6,2,\sigma\left(2\right)\right)=P\left(6,2,1\right)\end{eqnarray*}
Now, owing to 
$$\theta^{2}=\left(\begin{array}{cccccc}
1 & 2 & 3 & 4 & 5 & 6\\
2 & 3 & 1 & 6 & 4 & 5\end{array}\right)=\beta_3\diamond\mid \alpha_3,$$ we see that $\theta^2$ restricted to $P\left(6,2,1\right)$ and $P\left(6,2,2\right)$ are simple permutations ($\alpha_3$ or $\beta_3$). For instance, $\theta\in Sim(6)$. We can recover $\theta$ throughout $\theta^2$ using right pasting.
\end{ex}

The previous example gives a motivation to find the complete set of simple permutations with mixed order $6$ and $10$, using right Pasting. From now on we denote by $\sigma$ the transposition $(1,2)$, that is, $\sigma=(1,2)$. According to Theorem \ref{th9}, there are $12$ simple permutations of order $6$ in $4$ branches, each one of these branches correspond
 to $\theta^{2}$ by right Pasting of $\alpha_{3}$ and $\beta_{3}$ as follows:

\begin{enumerate}
\item Consider $\alpha=\alpha_3$ and $\theta_{\alpha\alpha}$ such that $\theta_{\alpha\alpha}^{2}=\alpha\diamond\mid\alpha$, i.e.,
$$\theta_{\alpha\alpha}^{2}=\left(\begin{array}{cccccc}
1 & 2 & 3 & 4 & 5 & 6\\
3 & 1 & 2 & 6 & 4 & 5\end{array}\right),$$ which give us the following options for $\theta_{\alpha\alpha}$.
\begin{eqnarray*}
\theta_{\alpha\alpha6} & = & \left(\begin{array}{cccccc}
1 & 2 & 3 & 4 & 5 & 6\\
6 & 4 & 5 & 1 & 2 & 3\end{array}\right)\\
\theta_{\alpha\alpha5} & = & \left(\begin{array}{cccccc}
1 & 2 & 3 & 4 & 5 & 6\\
5 & 6 & 4 & 2 & 3 & 1\end{array}\right)\\
\theta_{\alpha\alpha4} & = & \left(\begin{array}{cccccc}
1 & 2 & 3 & 4 & 5 & 6\\
4 & 5 & 6 & 3 & 1 & 2\end{array}\right)\end{eqnarray*}
For ${i=4,5,6}$ we have
\begin{eqnarray*}
\theta_{\alpha\alpha i}\left[P\left(6,2,1\right)\right] & = & \theta_{\alpha\alpha i}\left[\left\{ 1,2,3\right\} \right]\\
& = & \left\{ 4,5,6\right\}\\
& = & P\left(6,2,\sigma\left(1\right)\right)\\
& = & P\left(6,2,2\right)\end{eqnarray*}
\begin{eqnarray*}
\theta_{\alpha\alpha i}\left[P\left(6,2,2\right)\right] & = & \theta_{\alpha\alpha i}\left[\left\{ 4,5,6\right\} \right]\\
& = & \left\{ 1,2,3\right\}\\
& = & P\left(6,2,\sigma\left(2\right)\right)\\
& = & P\left(6,2,1\right)\end{eqnarray*}
Furthermore, $\theta^2_{\alpha\alpha i}$ restricted to $P\left(6,2,1\right)$ and $P\left(6,2,2\right)$ is $\alpha_3$, which is simple. For instance $\{\theta_{\alpha\alpha i}: 4\leq i\leq 6\}\subset Sim(6)$.

\item Consider $\beta=\beta_3$ and $\theta_{\beta\beta}$ such that $\theta_{\beta\beta}^{2}=\beta\diamond\mid\beta$, i.e.,
$$\theta_{\beta\beta}^{2}=\left(\begin{array}{cccccc}
1 & 2 & 3 & 4 & 5 & 6\\
2 & 3 & 1 & 5 & 6 & 4\end{array}\right),$$
which give us the following options for $\theta_{\beta\beta}$.
\begin{eqnarray*}
\theta_{\beta\beta 4} & = & \left(\begin{array}{cccccc}
1 & 2 & 3 & 4 & 5 & 6\\
4 & 5 & 6 & 2 & 3 & 1\end{array}\right)\\
\theta_{\beta\beta 5} & = & \left(\begin{array}{cccccc}
1 & 2 & 3 & 4 & 5 & 6\\
5 & 6 & 4 & 1 & 2 & 3\end{array}\right)\\
\theta_{\beta\beta 6} & = & \left(\begin{array}{cccccc}
1 & 2 & 3 & 4 & 5 & 6\\
6 & 4 & 5 & 3 & 1 & 2\end{array}\right)
\end{eqnarray*}
For ${i=4,5,6}$ we have
\begin{eqnarray*}
\theta_{\beta\beta i}\left[P\left(6,2,1\right)\right] & = & \theta_{\beta\beta i}\left[\left\{ 1,2,3\right\} \right]\\
& = & \left\{ 4,5,6\right\}\\
& = & P\left(6,2,\sigma\left(1\right)\right)\\
& = & P\left(6,2,2\right)\end{eqnarray*}

\begin{eqnarray*}
\theta_{\beta\beta i}\left[P\left(6,2,2\right)\right] & = & \theta_{\beta\beta i}\left[\left\{ 4,5,6\right\} \right]\\
& = & \left\{ 1,2,3\right\}\\
& = & P\left(6,2,\sigma\left(2\right)\right)\\
& = & P\left(6,2,1\right)\end{eqnarray*}
Furthermore, $\theta^2_{\beta\beta i}$ restricted to $P\left(6,2,1\right)$ and $P\left(6,2,2\right)$ is $\beta_3$, which is simple. For instance $\{\theta_{\beta\beta i}: 4\leq i\leq 6\}\subset Sim(6)$.

\item Consider $\alpha=\alpha_3$, $\beta=\beta_3$ and $\theta_{\alpha\beta}$ such that $\theta^{2}_{\alpha\beta}=\alpha\diamond\mid\beta$, i.e.,
$$\theta_{\alpha\beta}^{2}=\left(\begin{array}{cccccc}
1 & 2 & 3 & 4 & 5 & 6\\
3 & 1 & 2 & 5 & 6 & 4\end{array}\right).$$
Thus, we have the following options for $\theta_{\alpha\beta}$.
\begin{eqnarray*}
\theta_{\alpha\beta 4} & = & \left(\begin{array}{cccccc}
1 & 2 & 3 & 4 & 5 & 6\\
4 & 6 & 5 & 3 & 1 & 2\end{array}\right)\\
\theta_{\alpha\beta 5} & = & \left(\begin{array}{cccccc}
1 & 2 & 3 & 4 & 5 & 6\\
5 & 4 & 6 & 1 & 3 & 2\end{array}\right)\\
\theta_{\alpha\beta 6} & = & \left(\begin{array}{cccccc}
1 & 2 & 3 & 4 & 5 & 6\\
6 & 5 & 4 & 2 & 1 & 3\end{array}\right)\end{eqnarray*}
For ${i=4,5,6}$

\begin{eqnarray*}
\theta_{\alpha\beta i}\left[P\left(6,2,1\right)\right] & = & \theta_{\alpha\beta i}\left[\left\{ 1,2,3\right\} \right]\\
& = & \left\{ 4,5,6\right\}\\
& = & P\left(6,2,\sigma\left(1\right)\right)\\
& = & P\left(6,2,2\right)\end{eqnarray*}

\begin{eqnarray*}
\theta_{\alpha\beta i}\left[P\left(6,2,2\right)\right] & = & \theta_{\alpha\beta i}\left[\left\{ 4,5,6\right\} \right]\\
& = & \left\{ 1,2,3\right\}\\
& = & P\left(6,2,\sigma\left(2\right)\right)\\
& = & P\left(6,2,1\right)\end{eqnarray*}
Furthermore, $\theta^2_{\alpha\beta i}$ restricted to $P\left(6,2,1\right)$ and $P\left(6,2,2\right)$ is either $\alpha_3$ or $\beta_3$, which are simple. For instance $\{\theta_{\alpha\beta i}: 4\leq i\leq 6\}\subset Sim(6)$.
\item  Consider $\alpha=\alpha_3$, $\beta=\beta_3$ and $\theta_{\beta\alpha}$ such that $\theta^{2}_{\beta\alpha}=\beta\diamond\mid\alpha$, i.e.,
$$\theta_{\beta\alpha}^{2}=\left(\begin{array}{cccccc}
1 & 2 & 3 & 4 & 5 & 6\\
2 & 3 & 1 & 6 & 4 & 5\end{array}\right).$$ Thus, we have the following options for $\theta_{\beta\alpha}$

\begin{eqnarray*}
\theta_{\beta\alpha 4} & = & \left(\begin{array}{cccccc}
1 & 2 & 3 & 4 & 5 & 6\\
4 & 6 & 5 & 2 & 1 & 3\end{array}\right)\\
\theta_{\beta\alpha 5} & = & \left(\begin{array}{cccccc}
1 & 2 & 3 & 4 & 5 & 6\\
5 & 4 & 6 & 3 & 2 & 1\end{array}\right)\\
\theta_{\beta\alpha 6} & = & \left(\begin{array}{cccccc}
1 & 2 & 3 & 4 & 5 & 6\\
6 & 5 & 4 & 1 & 3 & 2\end{array}\right)\end{eqnarray*}
 For ${i=4,5,6}$ we have
\begin{eqnarray*}
\theta_{\beta\alpha i}\left[P\left(6,2,1\right)\right] & = & \theta_{\beta\alpha i}\left[\left\{ 1,2,3\right\} \right]\\
& = & \left\{ 4,5,6\right\}\\
& = & P\left(6,2,\sigma\left(1\right)\right)\\
& = & P\left(6,2,2\right)\end{eqnarray*}

\begin{eqnarray*}
\theta_{\beta\alpha i}\left[P\left(6,2,2\right)\right] & = & \theta_{\beta\alpha i}\left[\left\{ 4,5,6\right\} \right]\\
& = & \left\{ 1,2,3\right\}\\
& = & P\left(6,2,\sigma\left(2\right)\right)\\
& = & P\left(6,2,1\right)\end{eqnarray*}
Furthermore, $\theta^2_{\beta\alpha i}$ restricted to $P\left(6,2,1\right)$ and $P\left(6,2,2\right)$ is either $\alpha_3$ or $\beta_3$, which are simple. For instance $\{\theta_{\beta\alpha i}: 4\leq i\leq 6\}\subset Sim(6)$.
\end{enumerate}

For all of these simple permutations $\theta$, $\theta$(1) will be 4, 5 or 6. In general, if $\theta\in Sim(4n+2)$, then $\theta$(1) takes values between ${2n+2}$ and ${4n+2}$.

Following the previous procedure, it is possible to construct 20 simple permutations of order 10,  according to each one of the four branches of the genealogy associated to $Sim(10)$. Each branch corresponds to $\theta^{2}$ by right Pasting of $\alpha_{5}$ and $\beta_{5}$ as follows:

\begin{enumerate}

\item Consider $\alpha=\alpha_5$ and $\theta_{\alpha\alpha}$ such that $\theta_{\alpha\alpha}^{2} = \alpha\diamond\mid\alpha$. Thus, we have
$$\theta_{\alpha\alpha}^{2}=\left(\begin{array}{cccccccccc}
1 & 2 & 3 & 4 & 5 & 6 & 7 & 8 & 9 & 10\\
5 & 4 & 2 & 1 & 3 & 10 & 9 & 7 & 6 & 8\end{array}\right)$$ and the possible options for $\theta_{\alpha\alpha}$ are
\begin{eqnarray*}
\theta_{\alpha\alpha 6} & = & \left(\begin{array}{cccccccccc}
1 & 2 & 3 & 4 & 5 & 6 & 7 & 8 & 9 & 10\\
6 & 7 & 8 & 9 & 10 & 5 & 4 & 2 & 1 & 3\end{array}\right)\\
\theta_{\alpha\alpha 7} & = & \left(\begin{array}{cccccccccc}
1 & 2 & 3 & 4 & 5 & 6 & 7 & 8 & 9 & 10\\
7 & 10 & 6 & 8 & 9 & 2 & 5 & 1 & 3 & 4\end{array}\right)\\
\theta_{\alpha\alpha 8} & = & \left(\begin{array}{cccccccccc}
1 & 2 & 3 & 4 & 5 & 6 & 7 & 8 & 9 & 10\\
8 & 6 & 9 & 10 & 7 & 4 & 3 & 5 & 2 & 1\end{array}\right)\\
\theta_{\alpha\alpha 9} & = & \left(\begin{array}{cccccccccc}
1 & 2 & 3 & 4 & 5 & 6 & 7 & 8 & 9 & 10\\
9 & 8 & 10 & 7 & 6 & 3 & 1 & 4 & 5 & 2\end{array}\right)\\
\theta_{\alpha\alpha 10} & = & \left(\begin{array}{cccccccccc}
1 & 2 & 3 & 4 & 5 & 6 & 7 & 8 & 9 & 10\\
10 & 9 & 7 & 6 & 8 & 1 & 2 & 3 & 4 & 5\end{array}\right)\end{eqnarray*}
For ${i=6,7,8,9,10}$ we have
\begin{eqnarray*}
\theta_{\alpha\alpha i}\left[P\left(10,2,1\right)\right] & = & \theta_{\alpha\alpha i}\left[\left\{ 1,2,3,4,5\right\} \right]\\
& = & \left\{ 6,7,8,9,10\right\}\\
& = & P\left(10,2,\sigma\left(1\right)\right)\\
& = & P\left(10,2,2\right)\end{eqnarray*}
\begin{eqnarray*}
\theta_{\alpha\alpha i}\left[P\left(10,2,2\right)\right] & = & \theta_{\alpha\alpha i}\left[\left\{ 6,7,8,9,10\right\} \right]\\
& = & \left\{ 1,2,3,4,5\right\}\\
& = & P\left(10,2\sigma\left(2\right)\right)\\
& = & P\left(10,2,1\right)\end{eqnarray*}
Furthermore, $\theta^2_{\alpha\alpha i}$ restricted to $P\left(10,2,1\right)$ and $P\left(10,2,2\right)$ is  $\alpha_5$, which is simple. For instance $\{\theta_{\alpha\alpha i}: 6\leq i\leq 10\}\subset Sim(10)$.
\item Consider $\beta=\beta_5$ and $\theta_{\beta\beta}^{2}=\beta\diamond\mid\beta$ , that is, 
$$\theta_{\beta\beta}^{2}=\left(\begin{array}{cccccccccc}
1 & 2 & 3 & 4 & 5 & 6 & 7 & 8 & 9 & 10\\
3 & 5 & 4 & 2 & 1 & 8 & 10 & 9 & 7 & 6\end{array}\right).$$
The possible options for $\theta_{\beta\beta}$ are
\begin{eqnarray*}
\theta_{\beta\beta 6} & = & \left(\begin{array}{cccccccccc}
1 & 2 & 3 & 4 & 5 & 6 & 7 & 8 & 9 & 10\\
6 & 7 & 8 & 9 & 10 & 3 & 5 & 4 & 2 & 1\end{array}\right)\\
\theta_{\beta\beta 7} & = & \left(\begin{array}{cccccccccc}
1 & 2 & 3 & 4 & 5 & 6 & 7 & 8 & 9 & 10\\
7 & 8 & 10 & 6 & 9 & 2 & 3 & 5 & 1 & 4\end{array}\right)\\
\theta_{\beta\beta 8} & = & \left(\begin{array}{cccccccccc}
1 & 2 & 3 & 4 & 5 & 6 & 7 & 8 & 9 & 10\\
8 & 10 & 9 & 7 & 6 & 1 & 2 & 3 & 4 & 5\end{array}\right)\\
\theta_{\beta\beta 9} & = & \left(\begin{array}{cccccccccc}
1 & 2 & 3 & 4 & 5 & 6 & 7 & 8 & 9 & 10\\
9 & 6 & 7 & 10 & 8 & 5 & 4 & 1 & 3 & 2\end{array}\right)\\
\theta_{\beta\beta 10} & = & \left(\begin{array}{cccccccccc}
1 & 2 & 3 & 4 & 5 & 6 & 7 & 8 & 9 & 10\\
10 & 9 & 6 & 8 & 7 & 4 & 1 & 2 & 5 & 3\end{array}\right)\end{eqnarray*}
 For ${i=4,5,6}$ we have

\begin{eqnarray*}
\theta_{\beta\beta i}\left[P\left(10,2,1\right)\right] & = & \theta_{\beta\beta i}\left[\left\{ 1,2,3,4,5\right\} \right]\\
& = & \left\{ 6,7,8,9,10\right\}\\
& = & P\left(10,2,\sigma\left(1\right)\right)\\
& = & P\left(10,2,2\right)\end{eqnarray*}
\begin{eqnarray*}
\theta_{\beta\beta i}\left[P\left(10,2,2\right)\right] & = & \theta_{\beta\beta}\left[\left\{ 6,7,8,9,10\right\} \right]\\
& = & \left\{ 1,2,3,4,5\right\}\\
& = & P\left(10,2,\sigma\left(2\right)\right)\\
& = & P\left(10,2,1\right)\end{eqnarray*}
Furthermore, $\theta^2_{\beta\beta i}$ restricted to $P\left(10,2,1\right)$ and $P\left(10,2,2\right)$ is $\beta_3$, which is simple. For instance $\{\theta_{\beta\beta i}: 6\leq i\leq 10\}\subset Sim(10)$.

\item  Consider $\alpha=\alpha_5$, $\beta=\beta_5$ and $\theta_{\alpha\beta}^{2}=\alpha\diamond\mid\beta$. Thus, we have
$$\theta_{\alpha\beta}^{2}=\left(\begin{array}{cccccccccc}
1 & 2 & 3 & 4 & 5 & 6 & 7 & 8 & 9 & 10\\
5 & 4 & 2 & 1 & 3 & 8 & 10 & 9 & 7 & 6\end{array}\right).$$ The possible options for $\theta_{\alpha\beta}$ are 
\begin{eqnarray*}
\theta_{\alpha\beta 6} & = & \left(\begin{array}{cccccccccc}
1 & 2 & 3 & 4 & 5 & 6 & 7 & 8 & 9 & 10\\
6 & 7 & 9 & 10 & 8 & 5 & 4 & 3 & 2 & 1\end{array}\right)\\
\theta_{\alpha\beta 7} & = & \left(\begin{array}{cccccccccc}
1 & 2 & 3 & 4 & 5 & 6 & 7 & 8 & 9 & 10\\
7 & 8 & 6 & 9 & 10 & 2 & 5 & 4 & 1 & 3\end{array}\right)\\
\theta_{\alpha\beta 8} & = & \left(\begin{array}{cccccccccc}
1 & 2 & 3 & 4 & 5 & 6 & 7 & 8 & 9 & 10\\
8 & 10 & 7 & 6 & 9 & 1 & 2 & 5 & 3 & 4\end{array}\right)\\
\theta_{\alpha\beta 9} & = & \left(\begin{array}{cccccccccc}
1 & 2 & 3 & 4 & 5 & 6 & 7 & 8 & 9 & 10\\
9 & 6 & 10 & 8 & 7 & 4 & 3 & 1 & 5 & 2\end{array}\right)\\
\theta_{\alpha\beta 10} & = & \left(\begin{array}{cccccccccc}
1 & 2 & 3 & 4 & 5 & 6 & 7 & 8 & 9 & 10\\
10 & 9 & 8 & 7 & 6 & 3 & 1 & 2 & 4 & 5\end{array}\right)\end{eqnarray*}
For ${i=6,7,8,9,10}$ we have

\begin{eqnarray*}
\theta_{\alpha\beta i}\left[P\left(10,2,1\right)\right] & = & \theta_{\alpha\beta i}\left[\left\{ 1,2,3,4,5\right\} \right]\\
& = & \left\{ 6,7,8,9,10\right\}\\
& = & P\left(10,2,\sigma\left(1\right)\right)\\
& = & P\left(10,2,2\right)\end{eqnarray*}
\begin{eqnarray*}
\theta_{\alpha\beta i}\left[P\left(10,2,2\right)\right] & = & \theta_{\alpha\beta}\left[\left\{ 6,7,8,9,10\right\} \right]\\
& = & \left\{1,2,3,4,5\right\}\\
& = & P\left(10,2,\sigma\left(2\right)\right)\\
& = & P\left(10,2,1\right)\end{eqnarray*}
Furthermore, $\theta^2_{\alpha\beta i}$ restricted to $P\left(10,2,1\right)$ and $P\left(10,2,2\right)$ is either $\alpha_5$ or $\beta_5$, which are simple. For instance $\{\theta_{\alpha\beta i}: 6\leq i\leq 10\}\subset Sim(10)$.

\item Consider $\alpha=\alpha_5$, $\beta=\beta_5$ and $\theta_{\beta\alpha}^{2} = \beta\diamond\mid\alpha$. Thus,
$$\theta_{\beta\alpha}^{2}=\left(\begin{array}{cccccccccc}
1 & 2 & 3 & 4 & 5 & 6 & 7 & 8 & 9 & 10\\
3 & 5 & 4 & 2 & 1 & 10 & 9 & 7 & 6 & 8\end{array}\right)$$ and the possible options for $\theta_{\beta\alpha}$
are
\begin{eqnarray*}
\theta_{\beta\alpha 6} & = & \left(\begin{array}{cccccccccc}
1 & 2 & 3 & 4 & 5 & 6 & 7 & 8 & 9 & 10\\
6 & 7 & 10 & 8 & 9 & 3 & 5 & 2 & 1 & 4\end{array}\right)\\
\theta_{\beta\alpha 7} & = & \left(\begin{array}{cccccccccc}
1 & 2 & 3 & 4 & 5 & 6 & 7 & 8 & 9 & 10\\
7 & 10 & 9 & 6 & 8 & 2 & 3 & 1 & 4 & 5\end{array}\right)\\
\theta_{\beta\alpha 8} & = & \left(\begin{array}{cccccccccc}
1 & 2 & 3 & 4 & 5 & 6 & 7 & 8 & 9 & 10\\
8 & 6 & 7 & 9 & 10 & 5 & 4 & 3 & 2 & 1\end{array}\right)\\
\theta_{\beta\alpha 9} & = & \left(\begin{array}{cccccccccc}
1 & 2 & 3 & 4 & 5 & 6 & 7 & 8 & 9 & 10\\
9 & 8 & 6 & 10 & 7 & 4 & 1 & 5 & 3 & 2\end{array}\right)\\
\theta_{\beta\alpha 10} & = & \left(\begin{array}{cccccccccc}
1 & 2 & 3 & 4 & 5 & 6 & 7 & 8 & 9 & 10\\
10 & 9 & 8 & 7 & 6 & 1 & 2 & 4 & 5 & 3\end{array}\right)\end{eqnarray*}
For ${i=6,7,8,9,10}$
\begin{eqnarray*}
\theta_{\beta\alpha i}\left[P\left(10,2,1\right)\right] & = & \theta_{\beta\alpha i}\left[\left\{ 1,2,3,4,5\right\} \right]\\
& = & \left\{ 6,7,8,9,10\right\}\\
& = & P\left(10,2,\sigma\left(1\right)\right)\\
& = & P\left(10,2,2\right)\end{eqnarray*}
\begin{eqnarray*}
\theta_{\beta\alpha i}\left[P\left(10,2,2\right)\right] & = & \theta_{\beta\alpha i}\left[\left\{ 6,7,8,9,10\right\} \right]\\
& = & \left\{ 1,2,3,4,5\right\}\\
& = & P\left(10,2,\sigma\left(2\right)\right)\\
& = & P\left(10,2,1\right)\end{eqnarray*}
Furthermore, $\theta^2_{\beta\alpha i}$ restricted to $P\left(10,2,1\right)$ and $P\left(10,2,2\right)$ is either $\alpha_5$ or $\beta_5$, which are simple. For instance $\{\theta_{\beta\alpha i}: 6\leq i\leq 10\}\subset Sim(10)$.
\end{enumerate}

\section*{Final Remarks and Open Questions}

The paper \cite{abdulla} evidenced the importance of simple permutations with order $4n+2$ as a particular case of $(2n+1)2^s$ strong orbits in agreement with \cite{Minimal}. However, this paper followed the notion of simple orbits in agreement with \cite{Bernhardt}. Keeping in mind that Pasting and Reversing operations have been used to describe the genealogy of simple permutations with order a power of two in a recursive way and the first particular case of mixed order $4n+2$ in a constructive way, the use of those operations as a way to describe periodic orbits can lead to a new perspective. The following problems are related to this paper and can be source of future papers.

The first problem to be developed from this paper, is the extension of the found theorems in order $4n+2$ to the following order $8n+4$ and subsequently to the complete ``middle tail" of Sharkovskii's order. We recall that in order $8n+4$ there are two permutations $\sigma$ with order a power of two, and it would increase the ways to construct simple permutations with mixed order.

The second problem to be developed is to find the relation between Pasting and Reversing of simple permutations with order $4n+2$ and the dynamic of the associated primitive functions. Markov graphs shows some facts about this dynamic (for example, the existence of non-connected vertices and its relation with the existence of some periods, critical points, etc), but what can we say about Pasting and Reversing? How can be affected the dynamic whether we use pasting and reversing of simple permutations? What happens with Markov graphs and primitive functions whether we use pasting and reversing of simple permutations? Can we rewrite some results through of pasting of q-cycles?

Finally, this approach of Pasting and Reversing must be used to study the notion of primary orbits introduced by Alseda, Llibre and Misiurewicz, which plays the
same (and better) role as minimal orbits. In this context, this paper differs to the paper of Abdulla et. al. It must be interesting to apply this approach of Pasting and Reversing to recover the results concerning to strong simple orbits and digraphs.

% ------------------------------------------------------------------------

\subsection*{Acknowledgments} The first author is partially supported by the MICIIN/FEDER grant number MTM2012--31714, by the Generalitat de Catalunya grant number 2009SGR859 and by ECOS Nord France - Colombie C12M01. Finally, the authors thank to the anonymous referees by their valuable suggestions and comments.
% ------------------------------------------------------------------------

\end{document}